\documentclass[]{amsart}  

\usepackage[T1]{fontenc}
\usepackage{Nyquist_sty}
\usepackage{xcolor}
\usepackage{subcaption}

\usepackage{mathtools}
\usepackage{bbm}
\usepackage{booktabs}
\usepackage{lmodern}
\usepackage[]{todonotes}
\usepackage[unicode]{hyperref}

\PassOptionsToPackage{unicode}{hyperref}

\usepackage{tikz}
\usetikzlibrary{tikzmark, shapes.geometric, shapes.arrows, fit, bending}

\numberwithin{equation}{section}

\newcommand{\close}[1]{\overline{#1}}
\newcommand{\interior}[1]{#1^\mathrm{o}}
\newcommand\restr[2]{{%
  \left.\kern-\nulldelimiterspace %
  #1 %
  \vphantom{\big|} %
  \right|_{#2} %
  }}
\newcommand{\spt}{\operatorname{spt}}
\newcommand{\proj}{\operatorname{proj}}
\newcommand{\productofmarginals}{m}
\newcommand{\shiftedkpairsset}[1]{\overrightarrow{#1}}

\makeatletter
\newcommand*\bigcdot{\mathpalette\bigcdot@{.5}}
\newcommand*\bigcdot@[2]{\mathbin{\vcenter{\hbox{\scalebox{#2}{$\m@th#1\bullet$}}}}}
\makeatother
\renewcommand{\cdot}{\bigcdot}

\begin{document}   

\title[Large deviations for scaled Schr\"odinger bridges with reflection]{Large deviations for scaled families of Schr\"odinger bridges with reflection}

\author[V.~Nilsson]{Viktor Nilsson}
\author[P.~Nyquist]{Pierre Nyquist}
\address[V.~Nilsson, P.\ Nyquist]{KTH Royal Institute of Technology}
\address[P.\ Nyquist]{Chalmers University of Technology and University of Gothenburg}
\curraddr{Department of Mathematics, KTH, 100 44 Stockholm, Sweden}
\email{{vikn@kth.se}, {pnyquist@chalmers.se}}
\thanks{}



\begin{abstract} 

In this paper, we show a large deviation principle for certain sequences of static Schr\"{o}dinger bridges, typically motivated by a scale-parameter decreasing towards zero, extending existing large deviation results to cover a wider range of reference processes. 
Our results provide a theoretical foundation for studying convergence of such Schr\"{o}dinger bridges to their limiting optimal transport plans. 
Within generative modeling, Schr\"{o}dinger bridges, or entropic optimal transport problems, constitute a prominent class of methods, in part because of their %
computational feasibility in high-dimensional settings. Recently, Bernton, Ghosal, and Nutz \cite{bernton2022entropic} established a large deviation principle, in the small-noise limit, for fixed-cost entropic optimal transport problems. %
In this paper, we address an open problem posed in \cite{bernton2022entropic} and extend their results to hold for %
Schr\"{o}dinger bridges associated with certain sequences of more general reference measures with enough regularity in a similar small-noise limit. 
These can be viewed as sequences of entropic optimal transport plans with non-fixed cost functions. 
Using a detailed analysis of the associated Skorokhod maps and transition densities, we show that the new large deviation results cover Schr\"{o}dinger bridges where the reference process is a reflected diffusion on a bounded convex domain, corresponding to recently introduced model choices in the generative modeling literature.

\end{abstract}   

\maketitle

\section{Introduction}
\label{sec:intro}
\emph{Optimal transport} (OT) theory has come to play a central role in mathematics, bridging areas such as statistical physics, probability theory, analysis, and geometry; see, e.g., \cite{Vil09, rachev2006a, rachev2006b} and references therein. 
Recent interest in the area has been greatly spurred on by computational advances, where OT is now used to design and analyze methods for high-dimensional problems in machine learning and statistics --- the monograph \cite{peyre2019} contains numerous examples and references. 
A key element is the use of entropic regularization in the OT problem, thus studying \emph{entropic optimal transport} (EOT). 
This problem can be phrased as follows --- for ease of exposition, we here phrase the problem on $\bR ^d$ (see Section \ref{sec:SBs-and-EOT} for a more general definition and details): given a continuous cost function $c:\bR ^d \times \bR ^d \to \bR$ and two probability measures $\mu, \ \nu$ on $\bR ^d$, the EOT problem is, for $\varepsilon >0$,
\begin{align}
\label{eq:EOT}
    \inf _{\pi \in \Pi (\mu, \nu)} \int c \, d\pi + \varepsilon \calH (\pi \mid \mid \mu \times \nu),
\end{align}
where $\calH (\cdot \mid \mid \mu \times \nu )$ is the relative entropy with respect to the product measure $\mu \times \nu$ and $\Pi (\mu, \nu)$ is the set of couplings between $\mu$ and $\nu$. 
Considering the EOT problem as opposed to the original OT problem between $\mu$ and $\mu$, i.e., \eqref{eq:EOT} with $\varepsilon = 0$, enables the use of Sinkhorn's algorithm \cite{cuturi2013}, which in turn opens for applications in large-scale computational problems.  

From a theoretical perspective, the EOT problem can be phrased as a \emph{static Schr\"{o}dinger bridge} (SB) \emph{problem} for the same marginals $\mu$ and $\nu$. 
Dating back to work by Schr\"{o}dinger \cite{schrodinger1932theorie}, the static SB problem for $\mu$ and $\nu$, amounts to finding the coupling $\pi$ of $\mu$ and $\nu$ such that it minimizes the relative entropy with respect to some reference measures $R \in \mathcal{P} (\bR ^d \times \bR ^d)$; see Section \ref{sec:SBs-and-EOT} for details. %
The \textit{dynamic} version of the problem amounts to finding a stochastic process on $[0,1]$, identified through its path measure $\pi$, such that $\pi_0 = \mu$ and $\pi_1 = \nu$, while also minimizing $\calH(\cdot \mid\mid R)$ over all path measures that satisfy the marginal constraints (i.e., at times $0$ and $1$), where the reference measure $R$ is then also a measure on path space; we abuse terminology slightly and refer to this both as a reference measure and a reference \textit{dynamics or process}, as in the stochastic process governed by the path measure. 
The static problem can thus be retrieved from the dynamic formulation by projecting onto the marginals at times $0$ and $1$.
The link between EOT and SB problems is that a specific choice of cost function $c$ in the EOT setting, or equivalently, a specific choice of reference measure $R$ in the SB problem, makes the two problems equivalent.

In applications, the EOT problem \eqref{eq:EOT} is often used to replace the corresponding OT problem because the former can be efficiently solved numerically. 
As the aim is to have a good approximation to the solution of the OT problem, it is important to understand the convergence of the solution of the EOT problem as the regularization parameter $\varepsilon$ goes to zero.
In the pioneering work \cite{mikami2004monge}, Mikami proves the foundational result that the solution of the (dynamic) SB problem with a scaled Brownian reference process converges weakly to the OT plan between the given marginals $\mu$ and $\nu$ as $\varepsilon$, the noise-scale of the Brownian motion, goes to zero. 
L\'eonard extends this result in \cite{leonard2012schrodinger} to cover SBs with a sequence of general reference measures, where this sequence itself satisfies a certain convergence criterion, giving a general convergence result for SBs; these results are also for the dynamic setting. %
See also \cite{leonard2014survey} for a heuristic overview of L\'{e}onard's results in terms of what is there referred to as ``slowed down'' processes.

Although Mikami and L\'{e}onard study the convergence of the optimizers in the EOT setting, the convergence results in \cite{mikami2004monge, leonard2012schrodinger} do not provide any details about the \textbf{rate} of convergence. 
To address this, amongst other problems, in \cite{bernton2022entropic}, Bernton, Ghosal and Nutz establish a \emph{large deviation principle} (LDP) for the EOT problem in a small-noise limit. 
By the equivalence mentioned earlier, this also establishes an LDP for certain sequences of static SB problems. 
The results rely on studying the geometry of the optimizers in the EOT problem using the notion of $c$-cyclical monotonicity (see Section \ref{sec:large-deviations}). 
In \cite{Kato24}, Kato answers an open question posed in \cite{bernton2022entropic} and proves an analogous large deviation result in the setting of dynamic SBs with a (scaled) Brownian reference process.

From the perspective of SBs, the results in \cite{bernton2022entropic} apply only to the specific setting where the range of SB problems corresponding to different noise-scales $\varepsilon$ all can be represented as EOT problems with a common cost function $c$. %
The results in \cite{Kato24} consider the even more particular setting where the reference process must be a standard Brownian motion, which translates into $c$ in the corresponding EOT to be the quadratic cost function $c(x,y) = \frac{1}{2} |x-y| ^2$.

In this paper, we address a second open problem mentioned in \cite{bernton2022entropic}: establish an LDP analogous to that of \cite{bernton2022entropic} for more general sequences of reference measures $(R_\eta) _{\eta > 0}$. 
Here, $\eta$ is again viewed as a noise-scale in the reference process and we are interested in the limit as $\eta$ goes to zero. 
To align with the results in \cite{bernton2022entropic}, we formulate the general problem as a sequence of EOT problems where the cost function $c$ is no longer fixed but instead allowed to vary with the regularization parameter. 

Specifically, we provide a partial answer to the problem posed in \cite{bernton2022entropic}, proving an LDP for the optimizers in the sequence of SB problems in the case where the sequence of cost functions converges uniformly (in an appropriate sense) as the noise-scale goes to zero. 
The result is presented in Theorem \ref{thm:weak-type-LDP}. As a demonstration of the potential use of this theorem, we show that this generalization can be used to prove a large deviation principle for SBs where the reference process is a reflected Brownian motion. 
Whereas the extension of the results in \cite{bernton2022entropic} to cover the case of uniformly convergent cost functions is a straightforward modification of the arguments used for a common cost function, proving that the new result holds for the setting of a reflected reference process requires a detailed analysis of the Skorokhod maps and transition densities involved (see Section \ref{sec:reflected-SB}). 

The extension of the large deviation results in \cite{bernton2022entropic}, aside from their mathematical interest within optimal transport theory, is motivated by recent developments in generative modeling.  
The current wave of innovation in image-space generative models can largely be attributed to models using iterative refinement.
This is usually modeled by a continuous-time stochastic process that `connects' two probability distributions, often a structured data distribution and an unstructured noisy distribution.
Early work that explicitly used this framework suggested neural ODEs learned by likelihood training \cite{chen2018neural, grathwohl2018ffjord}. 
Seminal works in the field have used a denoising score-matching \cite{vincent2011connection, song2019generative} or diffusion model \cite{sohl2015deep, ho2020denoising} formulation, where an increasing sequence of noise-scales $\{ \sigma_t \}$ is used to corrupt the data, whereas a denoising neural model learns to remove this noise.
The increasing (in $t$) noise-scale sequence can be viewed as corresponding to the time variable $t$ of a stochastic differential equation (SDE) that continuously adds noise, and comes close in law to some normal distribution for a large enough $t$. %
Meanwhile, the denoiser learns to model the drift of the reversed SDE, which is used to generate new structured data.
This connection to SDEs was first explicitly made in \cite{song2020score}.

Another line of work in generative modeling, known as \textit{diffusion Schr\"{o}dinger bridges}, were introduced in \cite{bortoli2021diffusion, wang2021deep, vargas2021solving}. Diffusion SBs are attempts to model the true SB between two distributions with deep networks, by learning a target inspired by score-matching.
This directly couples the distributions, instead of relying on being approximately normal for large $t$.
Further, it allows both endpoints to have structure, which can be utilized for applications such as image to image translation \cite{shi2023diffusion}, and image restoration \cite{liu2023I2ISB}.
As these models rely on the dynamic SB problem, they require specifying a reference dynamics, often a scaled Brownian motion or Ornstein-Uhlenbeck process.
Taking this noise-scale to zero leads to the \emph{flow-matching}~\cite{lipman2022flow, albergo2022building, liu2022flow}~framework, where one learns a deterministic ODE that couples the two distributions.

Flow-matching models and diffusion SBs have strong intrinsic connections with OT problems. 
For flow matching, the simplest and most common choice of conditional vector field is the \emph{conditional OT} choice, consisting of straight transport paths \cite{lipman2022flow}. 
In \cite{liu2022flow}, under the name \emph{rectified flows}, it is shown that with this choice of conditional vector field, if the fitting procedure is iterated, the model converges in the limit to a \emph{dynamic} OT plan. %
In \cite{shi2023diffusion}, Shi et al.\ show how SBs can be learned by an iterative scheme similar to that of \cite{liu2022flow}, which then turns into the rectified flows of \cite{liu2022flow} in the low-noise limit.

In other works, such as \cite{pooladian2023multisample, tong2023improving}, the authors instead use OT couplings from mini-batches from the marginals, and then match the conditional vector fields produced by those couplings, which heuristically produces OT-like behavior.

The success and strengths of the different models are highly coupled with their OT characteristics. 
Empirically, this has been demonstrated by the proximity of generated samples to their origin \cite{bortoli2021diffusion, shi2023diffusion}, and the models' abilities to generate samples via straight, non-intersecting paths \cite{liu2022flow}, which may speed up the numerical solution of the generative ODE or SDE. 
Because of this, for models that can be phrased in terms of SBs, it is therefore of great interest to better understand the convergence to the limiting OT plan, when there is enough regularity in the problem, as the noise-scale goes to zero. 
This is precisely the role of the type of large deviation results proved here and in \cite{bernton2022entropic}. 
More specifically, from the point of view of applications, the results of this paper extend existing large deviation results to cover the newly established \emph{reflected Schr\"{o}dinger Bridges} \cite{caluya2021reflected, deng2024reflected}, a family of generative models for constrained domains which can be seen as the counterpart of reflected diffusion models \cite{lou2023reflected}.

The remainder of the paper is organized as follows. 
In Section \ref{sec:prel}, we provide some preliminaries on EOT and static SBs (Section \ref{sec:SBs-and-EOT}), dynamic SBs (Section \ref{sec:dynamic}), and large deviations (Section \ref{sec:large-deviations}), specifically reviewing the existing results for SBs in the latter. In Section \ref{sec:small-noise-sdes-rsdes-SBs} we introduce a family of SB problems with reference measures associated with certain small-noise diffusions. We illustrate how, when choosing the parameters involved in a certain way, this corresponds to a family of EOT problems with non-constant (in the parameter $\eta$) cost functions $\{ c_\eta \} _{\eta > 0}$. In Section \ref{sec:ldp-uniform-convergent-cost} we generalize the existing results on LDPs for EOT problems to the setting of such a non-constant family of costs $\{c_\eta\}_{\eta > 0}$ under a uniform convergence condition; the main result is Theorem \ref{thm:weak-type-LDP}, which is the analogous result to Theorem 1.1 in \cite{bernton2022entropic} (see Theorem \ref{thm:BGN22}). 
In Section \ref{sec:small-noise-reflected-browian-motion}, we provide some background on reflected diffusions and the corresponding SB problems. We then show that the condition of uniform convergence holds for the costs associated with reflected Brownian reference processes on smooth bounded domains, resulting in Theorem \ref{thm:reflected-ceta-uniform-convergence}. 
This establishes an LDP for the corresponding family of SBs.

\section{Preliminaries}
\label{sec:prel}
\subsection{Static Schr\"{o}dinger bridges and entropic optimal transport}\label{sec:SBs-and-EOT}

Starting from the rather informal and high-level discussion in Section \ref{sec:intro}, here we provide a more general, albeit brief, introduction to the static SB problem, EOT, and their connection. 

Throughout the paper, $\calX$ and $\calY$ are two Polish spaces (we will later take them to be subsets of $\bR ^d$), $c: \calX \times \calY \to \bR$ is a cost function, and $\mu$ and $\nu$ are elements of $\calP (\calX)$ and $\calP (\calY)$, the set of probability measures on $\calX$ and $\calY$, respectively. 
For $\varepsilon >0$, an $\varepsilon$-regularized EOT plan, or problem, is defined as
\begin{equation}\label{eq:EOT-def}
   \inf _{\pi \in \Pi(\mu, \nu)} \int_{\mathcal{X} \times \mathcal{Y}} c(x, y) \pi(dx, dy) + \varepsilon \calH(\pi \mid\mid \productofmarginals),
\end{equation}
where $\productofmarginals \coloneq \mu \times \nu$, $\Pi(\mu, \nu)$ denotes the set of couplings between $\mu$ and $\nu$, i.e., the set of measures $\pi \in \calP(\calX \times \calY)$ with marginals $\mu$ and $\nu$ (meaning that $(\proj_\calX)_\# \pi = \mu$ and $(\proj_\calY)_\# \pi = \nu$), %
and $\calH$ denotes the \emph{relative entropy}, or Kullback-Liebler (KL) divergence, 
\begin{align*}
\calH(\bP \mid\mid \bQ) = \begin{cases} \bE^\bP\left[ \log \frac{d\bP}{d\bQ} \right], & \textrm{if } \bP \ll \bQ, \\
\infty, & \textrm{otherwise}.
\end{cases}
\end{align*}
Taking $\varepsilon = 0$ in \eqref{eq:EOT-def} yields the original (OT) plan/problem associated with $\mu$ and $\nu$:
\begin{equation}\label{eq:OT-def}
    \inf_{\pi \in \Pi(\mu, \nu)} \int_{\mathcal{X} \times \mathcal{Y}} c(x, y) \pi(dx, dy).
\end{equation}
As mentioned in the introduction, the EOT problem can be viewed through the lens of SBs. 
For a \emph{reference measure} $R \in \calP(\calX \times \calY)$ the static SB problem between $\mu$ and $\nu$ with respect to $R$ is defined as the solution to 
\begin{equation}\label{eq:SB-def}
    \inf _{\pi \in \Pi(\mu, \nu)} \calH(\pi \mid\mid R).
\end{equation}
In \eqref{eq:EOT-def}-\eqref{eq:SB-def}, there is a question of whether or not the problems are well-posed and if there exists a (unique) minimizer. 
To address this, throughout the paper we make the following standard assumption (see, e.g., \cite{bernton2022entropic, Kato24}).
\begin{assumption}\label{ass:finite-plan-exists}
    In the static Schr\"{o}dinger problem, for all considered combinations of reference measures $R$ and marginals $\mu, \ \nu$, there exists at least one $\pi \in \Pi(\mu, \nu)$ such that $\calH(\pi \mid\mid R) < \infty$. Similarly, for the EOT problem, for any considered cost function $c$ and marginals $\mu, \ \nu$, there exists at least one $\pi \in \Pi(\mu, \nu)$ such that
    $\int_{\mathcal{X} \times \mathcal{Y}} c \, d\pi + \varepsilon \calH(\pi \mid\mid \productofmarginals) < \infty$.
\end{assumption}
Under Assumption \ref{ass:finite-plan-exists}, we have, by the the strict convexity of $\calH(\cdot \mid\mid \productofmarginals)$ \cite[Lemma~2.4b]{BD19} and compactness of $\Pi(\mu, \nu)$ \cite[Lemma~4.4]{Vil09}, that %
the SB problem \eqref{eq:SB-def} is guaranteed to have a well-defined minimizer, denoted $\pi^R$.
This is analogously true for the EOT problem \eqref{eq:EOT-def} with $\varepsilon > 0$, as the objective function is a sum of a linear and a strongly convex term; we denote the corresponding minimizer by $\pi _\varepsilon ^c$.
The OT problem \eqref{eq:OT-def} also has a minimizer under Assumption \ref{ass:finite-plan-exists}; however, uniqueness may not necessarily hold.

As alluded to in Section \ref{sec:intro}, if $R$ is absolutely continuous with respect to $\productofmarginals$, $R \ll \productofmarginals$, 
the SB problem \eqref{eq:SB-def} is equivalent to the $\varepsilon$-regularized EOT problem \eqref{eq:EOT-def} for a specific cost.
To see this, for a given $\varepsilon >0$, define 
\begin{equation}\label{eq:c_epsilon}
    c^\varepsilon(x,y) \coloneq -\varepsilon \log \frac{d R}{d \productofmarginals}(x,y).
\end{equation}
Inserting this $c^\varepsilon$ into \eqref{eq:EOT-def} 
leads to the two problems \eqref{eq:EOT-def} and \eqref{eq:SB-def} having the same minimizer,  $\pi^{c^\varepsilon} _\varepsilon = \pi^R$; note that $\pi ^{c^\varepsilon} _\varepsilon$ does not depend on $\varepsilon$ in this case.
Conversely, for any $\varepsilon$-regularized EOT problem with cost function $c$, define the measure $R_\varepsilon$ via
\begin{equation}\label{eq:R_epsilon}
    \frac{d R_\varepsilon}{d \productofmarginals} \coloneq Z_\varepsilon^{-1} e^{-\frac{1}{\varepsilon} c(x,y)},
\end{equation}
where $Z_\varepsilon$ is the normalizing constant. 
Taking this $R _\varepsilon$ as the reference measure in \eqref{eq:SB-def} produces a (static) SB problem that is equivalent to the EOT problem with cost function $c$, and the solution $\pi^{R_\varepsilon} = \pi^c_\varepsilon$.

\subsection{Dynamic Schr\"{o}dinger bridges}
\label{sec:dynamic}
Unsurprising, the study of Schr\"{o}dinger bridges in mathematics and physics goes back to Schr\"{o}dinger \cite{schrodinger1932theorie}, who conceived them as the answer to the question: 
What is the most likely way for a large number of non-interacting particles to evolve into a specified distribution at some fixed time $T > 0$?
The connection between this question and the SB problem \eqref{eq:SB-def} is perhaps easiest seen in the \emph{dynamic} formulation of the SB problem.

We begin with some notation. 
Take $\calC_1$ to be the path space $C([0,1]: \bR^d)$, i.e. the space of continuous functions from $[0,1]$ to $\bR ^d$, equipped with its Borel $\sigma$-algebra $\calB(\calC_1)$; $\calP (\calC _1)$ is the space of probability measures on $\calC_1$. Here we also take $\calX$ and $\calY$ to be subsets of $\bR^d$. 
For $t_1, ..., t_n \in [0, 1]$, let $\proj_{t_1, ..., t_n}: \calC_1 \to (\bR^d)^n$ be the projection $f \mapsto (f(t_0), ..., f(t_n)) \in (\bR^d)^n$, and, for a measure $\pi$ on $\calC_1$, let $\pi_{t_1 \, ... \, t_n} \coloneq ({\proj_{t_1, ..., t_n}})_\# \pi$. 
For $\mu, \nu \in \calP (\bR ^d)$, we say that a measure $\pi$ on $\calC_1$ is a \emph{path space coupling} between $\mu$ and $\nu$ if $\pi_0 = \mu$ and $\pi_1 = \nu$. 

In the dynamic SB problem, the reference measure $R$ is a measure on path space, i.e., here $R \in \calP(\calC_1)$. 
Let $\Pi^{\calC_1}(\mu, \nu)$ denote the set of path space couplings between $\mu$ and $\nu$,
\[
 \Pi^{\calC_1} (\mu, \nu) \coloneq \left\{ \pi \in \calP (\calC _1): \pi_0 = \mu, \ \pi_1 = \nu \right\}. %
\] 
The \emph{dynamic Schr\"{o}dinger bridge} with respect to $R$, $\mu$, and $\nu$ is given by 
\begin{equation}\label{eq:dynamic-SB-def}
    \hat{\pi} = \argmin_{\pi \in \Pi^{\calC_1}(\mu, \nu)} \calH(\pi \mid\mid R).
\end{equation}
In comparing the dynamic and static formulations, %
note that $R$ can be represented through the disintegration $R = R_{01} \otimes R^{\cdot}$, %
where $R^{\cdot}$ is an appropriate stochastic kernel, so that $R^{xy}$ is a path measure conditional on the endpoints $(x,y)$.
Moreover, for any measure $\pi \in \Pi^{\calC_1}(\mu, \nu)$, $\pi = \pi_{01} \otimes \pi^{\cdot}$.
Then, because $\pi \ll R$ implies that $\pi_{01} \ll R_{01}$ and $\pi^{xy} \ll R^{xy}$ hold $R_{01}$-$\mathrm{a.s.}$, we have that whenever $\frac{d\pi}{dR}(X)$ exists
\[
\frac{d\pi}{dR}(X) = \frac{d\pi_{01}}{dR_{01}}(X_0, X_1) \frac{d\pi^{X_0 X_1}}{dR^{X_0 X_1}}(X), \ \  R-\mathrm{a.s.},
\]
(it also holds $\pi$-$\mathrm{a.s.}$). 
From this, we have
\begin{equation}
\begin{split}
    \calH (\pi \mid\mid R) &= \bE ^{\pi} \left[ \log \frac{d \pi _{01} }{d R_{01}} (X_0, X_1) \right] + \bE ^{\pi} \left[ \log \frac{d \pi^{X_0 X_1} }{d R^{X_0 X_1}} (X) \right] \\
    &= \calH (\pi _{01} \mid\mid  R_{01}) + \int _{(\bR ^d)^2} \calH (\pi^{xy} \mid\mid R^{xy}) \pi _{01} (dx, dy).
\end{split}
\end{equation}
In cases where $\frac{d\pi}{dR}(X)$ does not exist, we have $\calH(\pi \mid \mid R) = \infty$. 
For $\pi _{01}$-a.s.\ every $(x,y) \in (\bR ^d)^2$, $\calH (\pi^{xy} \mid\mid R^{xy})$ is non-negative and zero iff $\pi^{xy} = R^{xy}$. Thus, the dynamic SB problem \ref{eq:dynamic-SB-def} amounts to the following static SB problem:
\begin{equation}
    \hat{\pi}_{01} \coloneq \argmin _{\pi_{01} \in \Pi(\mu, \nu)} \calH (\pi_{01} \mid\mid R_{01}),
\end{equation}
combined with interpolating the plans by $R^{\cdot}$, the bridge processes associated with $R$.
That is, the dynamic SB, the solution to \eqref{eq:dynamic-SB-def}, is given by the composition of  the static SB ($
\hat{\pi}_{01}$) with the conditional path process of $R$ ($R^{\cdot}$): $\hat{\pi} = \hat{\pi}_{01} \otimes R^{\cdot}$.

\subsection{Large deviations for EOT and static SB problems}
\label{sec:large-deviations}
In this section, we present the large deviation results shown in \cite{bernton2022entropic} for EOT problems, as the regularization parameter vanishes. We begin with a brief reminder of the concept that is at the heart of the theory of large deviations: the \textit{large deviation principle} (LDP). 
A sequence of probability measures $\{\gamma_\kappa\}_{\kappa > 0}$ on a Polish space $S$ is said to satisfy an LDP with %
\emph{rate function} $I: S \to [0, \infty]$, and speed $\kappa^{-1}$ if $I$ is lower semi-continuous and if for every Borel set $A \subseteq S$, the following inequalities hold:
\begin{align*}\label{eq:LDP-def}
    -\inf_{x \in \interior{A}} I(x) &\leq \liminf_{\kappa \downarrow 0} \kappa \log \gamma _\kappa(A) \\
    &\leq \limsup_{\kappa \downarrow 0} \kappa \log \gamma _\kappa (A) \leq -\inf_{x \in \close{A}} I(x),
\end{align*}
where $\interior{A}$ and $\close{A}$ denote the interior and closure of $A$, respectively; the definition can be made for much more general spaces than Polish, but for the purpose of this discussion, there is no need for such generalities. 
The rate function $I$ is said to be a \emph{good} rate function if the sub-level sets $I^{-1}([0,a])$, $a \geq 0$, are compact; see, e.g., \cite{Dembo98, BD19} and references therein for more details about large deviation theory in general.

The gist of the inequalities above is that they describe the exponential decay of probabilities under $\gamma _\kappa$ as $\kappa$ vanishes. 
In a rough sense, for an event $A \subseteq S$ and $\kappa$ small, %
\[
    \gamma _\kappa (A) \approx \exp \{-\kappa^{-1} \inf _{x\in A} I(x)\}.
\]
The definition of an LDP makes this approximation precise in the limit as $\kappa \to 0$. 
In particular, this implies that for the probability of an event $A$ not to vanish in this limit, it must hold that $\inf_{x \in A} I(x) = 0$. %
In the same vein, for any $A$, the set of elements $x \in A$ for which $I(x) \approx \inf_{x' \in A} I(x')$ is (asymptotically) the overwhelmingly most likely way $A$ can occur; this statement can be made rigorous, typically referred to as a Gibbs principle (see \cite{Dembo98, BD19}).

In the context of SBs and EOT, the  large deviation behavior of the small noise limit has only recently been studied in \cite{bernton2022entropic, Kato24}. 
Specifically, in \cite{bernton2022entropic}, Bernton, Ghosal and Nutz show the first LDP associated with a sequence of EOT problems as the regularization parameter vanishes. 
To review their results, we consider again the EOT problem \eqref{eq:EOT-def} for a given cost function and marginals $\mu \in \calP (\calX)$, $\nu \in \calP (\calY)$, 
\begin{align*}
   \inf _{\pi \in \Pi(\mu, \nu)} \int_{\mathcal{X} \times \mathcal{Y}} c(x, y) \pi(dx, dy) + \varepsilon \calH(\pi \mid\mid \productofmarginals), \ \ \productofmarginals = \mu \times \nu.
\end{align*}
In \cite{bernton2022entropic} the authors adopt a geometric point of view and a central tool used in their proofs is \textit{cyclical monotonicity} associated with the cost function $c$. 
\begin{definition}\label{def:cm}
    A subset $\Gamma \subseteq \calX \times \calY$ is called \textit{$c$-cyclically monotone} if
    \begin{equation}\label{eq:cyclical-monotonicity}
        \sum_{i=1}^k c(x_i, y_i) \leq \sum_{i=1}^k c(x_i, y_{i+1}),
    \end{equation}
    for all $k \in \bN$, $\{ (x_i, y_i ) \} _{i=1} ^k \subseteq \Gamma$, where $y_{k+1}$ is interpreted as $y_1$. 
    Correspondingly, a transport plan $\pi \in \Pi (\mu, \nu)$ is called cyclically monotone if $\pi(\Gamma) = 1$ for some cyclically monotone set $\Gamma$.
\end{definition}%

Under Assumption \ref{ass:finite-plan-exists}, the EOT problem above has a unique minimizer $\pi ^c _\ve \in \Pi (\mu, \nu)$. 
Moreover, in \cite[Proposition 2.2]{bernton2022entropic} it is shown that $\pi ^c _\ve$ can be the unique solution to the EOT problem if and only if it satisfies \textit{cyclical invariance} with respect to $c$ (and for the value $\ve$).
\begin{definition}
\label{def:ci}
An element $\pi \in \Pi (\mu, \nu)$ is called \textit{$(c,\ve)$-cyclically invariant} if $\pi$ is equivalent to $\productofmarginals$ and there is a version of the Radon-Nikodym derivative $d\pi/d\productofmarginals$ such that
\begin{align}
\label{eq:cyclical-invariance}
    \prod_{i=1}^k \frac{d \pi }{d \productofmarginals}(x_i, y_i) = \exp \left\{ -\frac{1}{\ve}\left(\sum_{i=1}^k c(x_i, y_i) - c(x_i, y_{i+1})\right) \right\} \prod_{i=1}^k \frac{d \pi}{d \productofmarginals}(x_i, y_{i+1}),    
\end{align}
for all $k \in \bN$, $\{ (x_i, y_i ) \} _{i=1} ^k \subseteq \calX \times \calY$, where $y_{k+1}$ is interpreted as $y_1$.
\end{definition}
See any of \cite{bernton2022entropic, ghosal2022stability, nutz2021introduction} for a proof, and further discussion, of the equivalence between $(c,\ve)$-cyclical invariance and the EOT. 
Cyclical monotonicity plays a similar role in characterizing (unregularized) OT plans: under weak assumptions, e.g., lower semi-continuity and non-negativity of $c$, and $\inf_{\pi \in \Pi(\mu, \nu)} \int c \, d\pi < \infty$, we have that $\pi \in \Pi(\mu, \nu)$ is cyclically monotone if and only if it is %
a (possibly non-unique) OT plan, see \cite{Vil09}. 

We are now ready to give the main result from \cite{bernton2022entropic}. 
For simplicity, we here assume that there exists a weak limit $\pi$ of the minimizers $\pi ^c _\ve$ in the EOT problem (see Section \ref{sec:ldp-uniform-convergent-cost} or \cite{bernton2022entropic} for more details), and let $\Gamma \coloneq \spt \pi$, where $\spt$ denotes the \emph{support} of a measure.
Let $\Sigma (k)$ denote the set of permutations of $\{1, \dots, k\}$ and consider a $c$-cyclically monotone set $\Gamma$. Define $I:\calX \times \calY \to [0, \infty]$ as
\begin{equation}\label{eq:rate-function-candidate-preliminaries}
    I(x,y) = \sup_{k \geq 2} \sup_{(x_i,y_i)_{i=2}^k \subseteq \Gamma} \sup_{\sigma \in \Sigma(k)} \sum_{i=1}^k c(x_i, y_i) - c(x_i, y_{\sigma(i)}),
\end{equation}
where $(x_1,y_1) = (x,y)$.
\begin{theorem}[Theorem 1.1 in \cite{bernton2022entropic}]
\label{thm:BGN22}
    ``%
    Let $\Gamma = \spt \pi$, where $\pi^c_\varepsilon \to \pi$ weakly, and define $I$ as in \eqref{eq:rate-function-candidate-preliminaries}. 
    \begin{itemize}
        \item[(a)] For any compact set $K \subseteq \calX \times \calY$,
        \[
            \limsup _{\ve \to 0} \ve \log \pi ^c _\ve (K) \leq -\inf _{(x,y) \in K} I(x,y).
        \]
        \item[(b)] Let Assumption \ref{ass:bernton44mod} hold and define the sets $\calX _0 = \proj _{\calX} \Gamma$ and $\calY _0 = \proj _\calY \Gamma$. 
        For any open set $G \subseteq \calX_0 \times \calY _0$,
        \[
            \liminf _{\ve \to 0} \ve \log \pi _\ve ^c (G) \geq - \inf_{(x,y) \in G} I(x,y). \text{''}
        \]
    \end{itemize}
\end{theorem}
Phrased in terms of static SBs, Theorem \ref{thm:BGN22} gives a (`weak-type', see Section \ref{sec:ldp-uniform-convergent-cost}) LDP for the sequence of minimizers of \eqref{eq:SB-def} where the reference measures $R_\ve$ are defined via $dR_\ve / d\productofmarginals \propto e^{-c/\ve} $ (see Section \ref{sec:SBs-and-EOT} for the derivation). 
That is, the measures are defined through the common cost function $c$ scaled by $\ve$. 
In this terminology, an open problem posed in \cite{bernton2022entropic} is to extend the LDP of Theorem \ref{thm:BGN22} to more general sequences $\{ R_\eta \} _{\eta > 0}$ of reference measures in the SB problem. 
As a first step to give a partial answer to this problem, in Section \ref{sec:ldp-uniform-convergent-cost}, we generalize the LDP of \cite{bernton2022entropic} to cover sequences of EOT plans associated with parametrized cost functions with sufficient regularity. 
We then show in Section \ref{sec:small-noise-reflected-browian-motion} how this naturally fits with using certain families of reference measures in the corresponding SB problems, in turn a natural generalization of the setup implicitly used in \cite{bernton2022entropic}.

\section{Small-noise families of Schr\"{o}dinger bridges}\label{sec:small-noise-sdes-rsdes-SBs}
As mentioned in the previous sections and outlined in Section \ref{sec:SBs-and-EOT}, the sequence of EOTs considered in \cite{bernton2022entropic} can be viewed as a sequence of SB problems with reference measures $R_\ve$ defined by $dR_\ve / d \productofmarginals \propto e^{-c/\ve}$. 
In the case $c(x,y) = |x-y| ^2 / 2$, with $\nu$ absolutely continuous and with finite relative entropy with respect to Lebesgue measure (this holds under Assumption \ref{ass:finite-plan-exists}), one way to think about each such problem is via the corresponding dynamic formulation (see Section \ref{sec:dynamic}): abusing notation somewhat, let $R_\ve \in \calP (\calC _1)$ be the path measure associated with the scaled Brownian motion $\{ X^\ve _t \} _{t \in [0,1]}$ started in $\mu$,
\[
    dX^\ve _t = \sqrt{\ve} dW_t, \ \ t\in [0,1],
\]
where $W$ is a standard $d$-dimensional Brownian motion, and $X^\ve_0 = X_0 $ has distribution $ \mu$, and is independent of $W$. %
The associated static SB problem becomes %
\[
    \inf _{\pi \in \Pi (\mu, \nu)} \calH (\pi  \mid \mid R_{\ve, 01}).
\]
From the transition kernel for a scaled Brownian motion, as we will see, this corresponds precisely to the EOT problem
\[
    \inf _{\pi \in \Pi (\mu, \nu)} \int \frac{1}{2}|x-y|^2 d\pi + \ve \calH (\pi \mid \mid m).
\]

The above example works, i.e., we can link the EOT problem to a reference measure connected to a stochastic process, because the Brownian dynamics imply
\[
    R_{\ve, 01} (dx,dy) = \frac{1}{(2\pi \ve) ^{-d/2}} \exp \left\{ -\frac{|x-y| ^2}{2\ve} \right\} \mu (dx) dy.
\]
Therefore, taking as in \eqref{eq:c_epsilon}, $c^\ve (x,y) =  - \ve \log (dR_{\ve, 01} /d\productofmarginals) (x,y)$ returns back $|x-y|^2/2$ (plus an additive term that does not matter, as we will see in Section \ref{sec:ceta-explicit-form}), and thus the sequence of SB problems with $R_{\ve, 01}$ as reference measure is equivalent to the sequence of EOT problems with this quadratic cost.

The above example highlights the following: For SBs, although there is a priori no parameter to take to zero as in \eqref{eq:EOT-def}, for specific model choices there are often hyperparameters that can be interpreted as a scaling parameter, similar to $\varepsilon$ in \eqref{eq:EOT-def}.
For example, for $\eta >0$, under mild regularity conditions, one may consider $R_\eta$ to be the path measure on $\calC_1$ associated with the SDE
\begin{equation}\label{eq:small-scale-diffusion-drift-f}
\begin{split}
    dX^\eta_t &= f(t, X^\eta_t) dt + \sqrt{\eta} dW_t, \ \ t \in [0, 1], \\ 
    X_0 ^\eta &= X_0 \sim \mu.
\end{split}
\end{equation}
Then, one has a family of reference processes $\{ R_\eta \}_{\eta > 0}$, and correspondingly a family of SBs $\{ \pi_\eta \}_{\eta > 0} \coloneq \{\pi^{R_\eta}\}_{\eta > 0}$.
Assuming, e.g., a uniform Lipschitz condition on $f$, one can show that \eqref{eq:small-scale-diffusion-drift-f} converges weakly to a deterministic limit ODE given by $\eta = 0$ in \eqref{eq:small-scale-diffusion-drift-f} as $\eta \downarrow 0$.

\subsection{Equivalent cost sequences \texorpdfstring{$c_\eta$}{c\_eta}}\label{sec:ceta-explicit-form}
With a family $\{R_\eta \}_{\eta > 0}$, as in Section \ref{sec:small-noise-sdes-rsdes-SBs}, we can follow the procedure in the beginning of Section \ref{sec:SBs-and-EOT} and for each $\eta$ (and $\varepsilon$) introduce a cost function 
\begin{equation}\label{eq:c_eta_epsilon_def}
    c_\eta^\varepsilon(x, y) \coloneq -\varepsilon \log \left(\frac{dR_\eta}{d\productofmarginals}\right)(x,y) \quad \eta, \varepsilon > 0.
\end{equation}
\noindent which, used in \eqref{eq:EOT-def}, gives an equivalent EOT problem when using an $\varepsilon$ regularization parameter.
Note that $\varepsilon > 0$ (in \eqref{eq:c_eta_epsilon_def} and \eqref{eq:EOT-def}) can be chosen freely;
in particular, there is that nothing stops us from taking $\varepsilon = \eta$. We therefore define, and henceforth are only concerned with, the cost%
\begin{equation}\label{eq:c_eta_def}
    c_\eta \coloneq c_\eta^\eta = -\eta \log \left(\frac{dR_\eta}{dm}\right).
\end{equation}
The relevant EOT formulation of the SB problem is then
\begin{equation}\label{eq:EOT-eta}
    \pi_\eta %
    = \argmin_{\pi \in \Pi(\mu, \nu)} \int_{\mathcal{X} \times \mathcal{Y}} c_\eta \, d\pi + \eta \calH(\pi \mid\mid \productofmarginals), \quad \eta > 0.
\end{equation}
Notice also that for all $\eta$, by \eqref{eq:cyclical-invariance}, $\pi_\eta$ has the cyclical invariance characterization
\begin{equation}\label{eq:cyclical-invariance-c_eta}
    \prod_{i=1}^k \frac{d \pi_\eta}{d \productofmarginals}(x_i, y_i) = \exp -\frac{1}{\eta}\left[\sum_{i=1}^k c_\eta(x_i, y_i) - c_\eta(x_i, y_{i+1})\right] \prod_{i=1}^k \frac{d \pi_\eta}{d \productofmarginals}(x_i, y_{i+1}).
\end{equation}

In the primary situation of interest for us, $R_\eta$ is given by a time-homogeneous Markov process on $\bR^d$, started in $\mu$, such as \eqref{eq:small-scale-diffusion-drift-f}. %
We assume that its transition kernel $\kappa_\eta$, given by $\kappa_\eta(t, x, A) = R_\eta(X_t \in A \mid X_0 = x)$, admits a density $q_\eta(t, x, y)$, with respect to the Lebesgue measure $\lambda$ on $\bR^d$.
In particular, this is true of (reflected) Brownian motion.
For the analysis to be possible under these assumptions, we must also assume that $\nu \ll \lambda$. %
One gets that $R_{\eta, 01} = R_{\eta, 0} \otimes \kappa_\eta(1, \cdot, \cdot) = \mu \otimes \kappa_\eta(1, \cdot, \cdot)$~and
\begin{equation}\label{eq:dReta01dmxy-qeta}
    \frac{dR_{\eta, 01}}{d \productofmarginals}(x, y) %
    = \frac{d(\mu \otimes \kappa_\eta(1, \cdot, \cdot))}{d(\mu \times \nu)}(x, y) = \frac{d\kappa_\eta(1, x, \cdot)}{d\nu}(y) = \frac{q_\eta(1, x, y)}{\frac{d\nu}{d\lambda}(y)},
\end{equation}
and that $c_\eta(x,y) = - \eta \log q_\eta(1, x, y) - \eta \log \frac{d\nu}{d\lambda}(y)$.
This leads to the EOT~problem %
\begin{align*}
     \pi_\eta &= \argmin_{\pi \in \Pi(\mu, \nu)} \int c_\eta \, d\pi + \eta \calH(\pi \mid\mid \productofmarginals) \\
    & = \argmin_{\pi \in \Pi(\mu, \nu)} \int_{\calX^2} - \eta \log q_\eta(1, x, y) \pi(dx, dy) + \eta \calH(\nu) + \eta \calH(\pi \mid\mid \productofmarginals), \\
\end{align*}
where $\calH(\nu)$ denotes the differential entropy of $\nu$ which we will assume to be finite. 
We ignore this term since it is constant in $\Pi(\mu, \nu)$ and therefore does not affect the minimizer of the EOT problem.
Effectively, we may define (abusing notation),
\begin{equation}\label{eq:c_eta_transition_function}
    c_\eta(x, y) \equiv -\eta \log q_\eta(1, x, y) \quad \eta > 0,
\end{equation}
and use this in \eqref{eq:EOT-eta}.
This useful form will be used, exclusively, in the investigation of reflected Brownian motion in Section \ref{sec:small-noise-reflected-browian-motion}.
Moreover, \eqref{eq:dReta01dmxy-qeta}-\eqref{eq:c_eta_transition_function} also clarify that, under conditions of well-posedness, the only relevant part of $R_{\eta,01}$ to the static SB is its transition density $q_\eta(1, x, y)$. %

For \eqref{eq:EOT-eta} to be equivalent to \eqref{eq:EOT-def} with $\varepsilon = \eta$ (for some $c$) \textbf{across all $\eta$}, i.e., for $\{ \pi_\eta \}_{\eta > 0}$ and $\{ \pi^c_\varepsilon \}_{\varepsilon > 0}$ to be identical, we must have $c_\eta \equiv c$, i.e., that $c_\eta$ does not vary (modulo additive constants) with $\eta$.
This does not hold in much generality.
However, one important example where it does is the one used at the beginning of this section: when using (scaled) standard Brownian motion, i.e., $f = 0$ in \eqref{eq:small-scale-diffusion-drift-f}.
Then, as we have already seen, we have $c_\eta(x,y) \equiv \frac{1}{2} |x-y|^2$. 
To see this also from the current perspective, notice that the associated transition density is given by $q_\eta(t,x,y) = (2\pi \eta t)^{-d/2} \allowbreak \exp(-\frac{1}{2 \eta t} \allowbreak |x-y|^2)$, and that any additive constant may be disregarded in \eqref{eq:c_eta_transition_function}.

\section{LDP for uniformly convergent costs}\label{sec:ldp-uniform-convergent-cost}
In this section, we extend the results of \cite{bernton2022entropic} to the case when $\{c_\eta\}_{\eta > 0}$, given by \eqref{eq:c_eta_def} when starting from a general sequence $\{ R _\eta \} _{\eta >0}$ of reference measures in the SB problem, converges uniformly as $\eta \downarrow 0$, to some continuous function $c$. %
Specifically, we show that the large deviation results of \cite{bernton2022entropic} persist in the following more general setting: for all $\eta > 0$, $\pi_\eta$ is defined, on the Polish product space $\calX \times \calY$, by the EOT problem
\begin{equation}\label{eq:EOT-eta-2}
    \pi_\eta \coloneq \argmin_{\pi \in \Pi(\mu, \nu)} \int_{\mathcal{X} \times \mathcal{Y}} c_\eta d\pi + \eta \calH(\pi \mid\mid \productofmarginals),
\end{equation}
cf. \eqref{eq:EOT-eta}, where $\{c_\eta\}_{\eta > 0}$ converges uniformly to a continuous function $c: \calX \times \calY \to \bR$ as $\eta \downarrow 0$; we also assume the existence of $a \in L^1(\mu)$, $b \in L^1(\nu)$ such that $c(x,y) \geq a(x) + b(y)$ (e.g., we can simply assume that $c$ is non-negative).
Similar to before, the ``new'' EOT problem \eqref{eq:EOT-eta-2} can be interpreted as a sequence of SB problems via \eqref{eq:R_epsilon}.
The goal is to show a natural generalization of Theorem \ref{thm:BGN22} to cover the type of sequences of cost functions $\{ c_\eta \} _{\eta >0}$ considered here; this is formulated in Theorem \ref{thm:weak-type-LDP} below.

We assume that $\pi_\eta \to \pi$ weakly for some $\pi \in \Pi(\mu, \nu)$, which necessarily solves %
the OT-problem, 
\begin{equation}\label{eq:OT-c}
    \min_{\pi \in \Pi(\mu, \nu)} \int_{\calX \times \calY} c \, d\pi,
\end{equation}
see the discussion after Proposition~\ref{prop:bernton32mod}.
Let $\Gamma \coloneq \spt \pi$, $\calX_0 \coloneq \operatorname{proj}_\calX \Gamma$, and $\calY_0 \coloneq \operatorname{proj}_\calY \Gamma$.
Note that $\close{\calX_0} = \spt \mu$ and $\close{\calY_0} = \spt \nu$.%

In establishing the main result of this section, similar to in \cite{bernton2022entropic}, we need a solution $\psi: \calX \to (-\infty, \infty]$ to the \emph{dual Kantorovich problem} 
\begin{equation}\label{eq:kantorovich-dual-problem}
    \sup_{\psi \in L^1(\mu)} \int_{\calX} -\psi\, d\mu + \int_{\calY} \psi^c \, d\nu,%
\end{equation}
where the \emph{$c$-transform} $\psi^c: \calY \to [-\infty, \infty)$ of $\psi$ is given by $\psi^c(y) \coloneq \sup_{x \in \calX} c(x,y) + \psi(x)$.
The function $\psi$ is called \emph{$c$-convex} if there exists a $\phi: \calY \to [-\infty, \infty]$ such that $\psi(x) = \phi^c(x) \coloneq \inf_{y \in \calY} (\phi(y) - c(x,y))$.
The \emph{$c$-subdifferential} of a $c$-convex function $\psi$ is given by $\partial_c \psi \coloneq \{(x,y) \in \calX \times \calY: -\psi(x) + \psi^c(y) = c(x,y)\}$.
We call such a $\psi$ a \emph{Kantorovich potential} if $\Gamma \subseteq \partial_c \psi$.
Any Kantorovich potential is optimal for \eqref{eq:kantorovich-dual-problem}, which also coincides with \eqref{eq:OT-c}.
The connection to the quantitites studied so far is that \eqref{eq:kantorovich-dual-problem} is the dual problem of the OT problem, for $c$, $\mu$ and $\nu$, and the supremum equals the OT infimum. %
By $\pi$ solving the OT problem \eqref{eq:OT-c}, $\Gamma$ is cyclically monotone, and a Kantorovich potential is known to exist \cite[p. 65]{Vil09}.
See \cite[Ch. 5]{Vil09} for more on the OT duality theory.
The following assumption will be used in addition to Assumption \ref{ass:finite-plan-exists}.
\begin{assumption}[cf. Assumption 4.4 in \cite{bernton2022entropic}]\label{ass:bernton44mod}
    \emph{Uniqueness of Kantorovich potentials} holds,
    meaning that for any two Kantorovich potentials $\psi_1$ and $\psi_2$, $\psi_1 - \psi_2$ is constant on $\calX_0$. %
\end{assumption}
\noindent Assumption \ref{ass:bernton44mod} holds, for example, if $\interior{\calY}$ is connected, $\nu \ll \lambda$ (Lebesgue measure), and $c(x, \cdot)$ is differentiable for all $x$ and locally Lipschitz uniformly in $x$, see \cite{bernton2022entropic}, and also \cite{santambrogio2015optimal, Kato24} for a different condition.

\begin{theorem}\label{thm:weak-type-LDP}
    Let assumptions \ref{ass:finite-plan-exists} and \ref{ass:bernton44mod} hold.
    If $c_\eta \to c$ uniformly over $\mathcal{X} \times \mathcal{Y}$ as ${\eta \downarrow 0}$, then we get that the sequence $\{ \pi_\eta \}_{\eta > 0}$ satisfies a ``weak-type'' LDP with rate function $I: \calX \times \calY \to [0, \infty]$ given by %
    \begin{equation}\label{eq:rate-function-candidate-theorem}
        I(x, y) \coloneq \sup_{k \geq 2} \sup_{(x_i, y_i)_{2=1}^k \subseteq \Gamma} \sup_{\sigma \in \Sigma(k)} \sum_{i=1}^k c(x_i, y_i) - \sum_{i=1}^k c(x_i, y_{\sigma(i)}),
    \end{equation}
    This means that we have:
    \begin{enumerate}
        \item 
        For any open set $G \subseteq \calX_0 \times \calY_0$,
        \begin{equation*}
            -\inf_{(x,y) \in G} I(x,y) \leq \liminf_{\eta \downarrow 0} \eta \log \pi_\eta(G).
        \end{equation*}
        \item For any compact set $K \subseteq \calX \times \calY$,
        \begin{equation*}
            \limsup_{\eta \downarrow 0} \eta \log \pi_\eta(K) \leq -\inf_{(x,y) \in K} I(x,y).
        \end{equation*}
    \end{enumerate}
    Further, for any Kantorovich potential $\psi$, $I = c - (-\psi \oplus \psi^c)$ on $\calX_0 \times \calY_0$.
    
\end{theorem}
\noindent Above, we used the notation $(a \oplus b)(x,y) \coloneq a(x) + b(y)$.
The term ``weak-type'' LDP is borrowed from \cite{Kato24}, and refers to the fact that Theorem \ref{thm:weak-type-LDP} is not a full LDP. In fact, it is not even a weak LDP according to the standard definition.
However, it yields a full LDP under a compactness assumption on the supports, as highlighted by the following corollary; a brief proof is given at the end of this section. 
\begin{corollary}\label{cor:katocor31mod}
    In addition to the assumptions of Theorem \ref{thm:weak-type-LDP}, assume that %
    $\spt \mu$ and $\spt \nu$ are both compact.
    Then $\{\pi_\eta\}_{\eta > 0}$ satisfies a (full) LDP on $\calX \times \calY$ with the rate function %
    \begin{equation}\label{eq:rate-function-extended-cor}
        J(x,y) = \begin{cases}
            (c - (-\psi \oplus \psi^c))(x,y), &(x,y) \in \spt \mu \times \spt \nu, \\
            \infty, &(x,y) \in (\calX \times \calY) \setminus (\spt \mu \times \spt \nu).
        \end{cases}
    \end{equation}
\end{corollary}

\begin{remark}
    The criterion of uniform convergence of $c_\eta(x,y) = -\eta \log \frac{dR _\eta}{dm}(x,y)$ to $c$ may be interpreted in the broadest sense as $\|c_\eta + a_\eta \oplus b_\eta - c\|_\infty \to 0$ as $\eta \downarrow 0$, where $a_\eta: \calX \to \bR$ and $b_\eta: \calY \to \bR$ are $\mu$- and $\nu$-integrable sequences of functions, respectively, since the addition of any such term $a_\eta \oplus b_\eta$ does not affect the minimizer of the EOT problem.
\end{remark}

Consider the function $I$ in \eqref{eq:rate-function-candidate-theorem}.
$I(x,y)$ can be interpreted as the maximal amount of improvement (per unit mass) that can be made to a plan that includes $(x,y)$.
By cyclical monotonicity, for any point $(x,y)$ in the support of a $c$-optimal transport plan, $I(x,y)=0$.
We have that $I$ is non-negative, lower semicontinuous, and essentially equal to $I'$ given below; see \cite{bernton2022entropic} for more discussion about this.
With $(x_1,y_1) = (x,y)$, and $y_{k+1}$ is interpreted as $y_1$,
\begin{equation}\label{eq:rate-function-candidate-prime}
    I'(x,y) \coloneq \sup_{k \geq 2} \sup_{(x_i,y_i)_{i=2}^k \subseteq \Gamma} \sum_{i=1}^k c(x_i, y_i) - c(x_i, y_{i + 1}).
\end{equation}

We proceed to the proof of Theorem \ref{thm:weak-type-LDP}. %
From \cite{bernton2022entropic}, only minor modifications are needed to show Theorem~\ref{thm:weak-type-LDP}. However, for the paper to remain self-contained and for the ease of the reader, all the results we need from \cite{bernton2022entropic} (i.e., results \ref{lem:bernton31mod}-\ref{cor:bernton47mod}) are restated here. We also provide the adjusted proofs; no proofs are provided for Proposition \ref{prop:bernton45mod} and Corollary \ref{cor:bernton47mod}. %

\begin{remark}
    Before stating the claims and associated proofs of this section, we point out where we have diverged from the presentation in \cite{bernton2022entropic}.
    Lemma \ref{lem:bernton31mod} requires taking a $\delta_- < \delta$, to allow for the supremum difference $\|c_\eta - c\|_\infty$ to be handled as in \eqref{eq:bernton31mod-2}. This additional slack then requires the result to be restated slightly compared to the corresponding version in \cite{bernton2022entropic} (Lemma 3.1 therein), which is here done using a $\limsup$-$\log$-formulation (see \eqref{eq:bernton31mod-limsup}).
    Lemma \ref{lem:bernton41mod} is then based on thisslightly modified formulation, which leads to a similar modification and restatement compared to Lemma 4.1 in \cite{bernton2022entropic}.
    We also strengthened this result somewhat in another direction, by directly comparing with the rate function, yielding (almost) the large deviation upper bound for small balls.
    The extension of this result to compact sets, i.e., the large deviation upper bound in Corollary \ref{cor:bernton43mod}, can be done identically to the proof of Corollary 4.3 in \cite{bernton2022entropic}. Still, we find a somewhat different, original proof, following nicely from the way Lemmas \ref{lem:bernton31mod}-\ref{lem:bernton41mod} are formulated.
    Proposition \ref{prop:bernton32mod} is identical to Proposition 3.2 in \cite{bernton2022entropic} and follows from Lemma \ref{lem:bernton31mod} with similar adjustments. We omit the proofs of \ref{prop:bernton45mod} and \ref{cor:bernton47mod} as they can be carried out precisely as in \cite{bernton2022entropic}.
\end{remark}

\begin{lemma}[cf. Lemma 3.1 in \cite{bernton2022entropic}]\label{lem:bernton31mod}
    For $\delta, \delta' \in [0, \infty]$ with $\delta \leq \delta'$, let $A_k(\delta, \delta') \coloneq \{(x_i, y_i)_{i=1}^k \in (\calX \times \calY)^k: \sum_{i=1}^k c(x_i, y_i) - \sum_{i=1}^k c(x_i, y_{i + 1}) \in [\delta, \delta']\}$, where $y_{k+1}$ is interpreted as $y_1$. %
    Then, for any Borel set $A \subseteq A_k(\delta, \delta')$,
    \begin{equation}\label{eq:bernton31mod-limsup}
        \limsup_{\eta \downarrow 0} \eta \log \pi_\eta^k(A) \leq -\delta,
    \end{equation}
    and if $\shiftedkpairsset{A} \coloneq \{(x_i, y_{i+1})_{i=1}^k: (x_i, y_{i})_{i=1}^k \in A\}$ satisfies $\liminf_{\eta \downarrow 0} \eta \log \pi_\eta^k(\shiftedkpairsset{A}) = 0$, then also,
    \begin{equation}
        \liminf_{\eta \downarrow 0} \eta \log \pi_\eta^k(A) \geq -\delta'.
    \end{equation}
\end{lemma}

\begin{proof}
    We have by cyclical monotonicity that
    \begin{equation}\label{eq:bernton31mod-1}
        \prod_{i = 1}^k \frac{d \pi_\eta}{d \productofmarginals}(x_i, y_i) = \exp\left\{-\frac{1}{\eta} \sum_{i=1}^k (c_\eta(x_i, y_i) - c_\eta(x_i, y_{i+1})) \right\} \prod_{i = 1}^k \frac{d \pi_\eta}{d \productofmarginals}(x_i, y_{i+1})
    \end{equation}
    for some version of $\frac{d\pi_\eta}{dm}$. %
    Consider any $\delta_- < \delta$ and take $\eta_0 > 0$ such that $\|c_\eta - c\|_\infty \leq \frac{\delta - \delta_-}{2k}$ for any $\eta \leq \eta_0$.
    Then for all $\eta \leq \eta_0$ and $(x_i, y_i)_{i=1}^k \in A$,
    \begin{equation}\label{eq:bernton31mod-2}
    \begin{split}
        \sum_{i=1}^k (c_\eta(x_i, y_i) - c_\eta(x_i, y_{i+1})) &\geq \sum_{i=1}^k (c(x_i, y_i) - \frac{\delta - \delta_-}{2k}) - \sum_{i=1}^k (c(x_i, y_{i+1}) + \frac{\delta - \delta_-}{2k}) \\
        &= \sum_{i=1}^k c(x_i, y_i) - \sum_{i=1}^k c(x_i, y_{i+1}) - (\delta - \delta_-) \\
        &\geq \delta - (\delta - \delta_-) = \delta_-,
    \end{split}
    \end{equation}
    so \eqref{eq:bernton31mod-1} is $\prod_{i = 1}^k \frac{d \pi_\eta}{d \productofmarginals}(x_i, y_i) \leq e^{-\frac{\delta_-}{\eta}} \prod_{i = 1}^k \frac{d \pi_\eta}{d \productofmarginals}(x_i, y_{i+1})$.
    Thus, when integrating over $A$ with respect to $\productofmarginals^k$, we get
    \begin{equation}
        \pi_\eta^k(A) \leq e^{-\frac{\delta_-}{\eta}} \int_A \prod_{i = 1}^k \frac{d \pi_\eta}{d \productofmarginals}(x_i, y_{i+1}) d\productofmarginals^k = e^{-\frac{\delta_-}{\eta}} \productofmarginals^k(\shiftedkpairsset{A}) \leq e^{-\frac{\delta_-}{\eta}}.
    \end{equation}
    This means that $\limsup_{\eta \downarrow 0} \eta \log \pi_\eta^k(A) \leq -\delta_-$ and further $\limsup_{\eta \downarrow 0} \eta \log \pi_\eta^k(A) \leq -\delta$, since the former inequality holds for all $\delta_- \leq \delta$.
    TThis establishes the first part of the lemma.

    Taking $\delta_+ > \delta'$, one gets analogously, for a small enough $\eta_0$,  
    \begin{equation}
        \pi_\eta^k(A) \geq e^{-\frac{\delta_+}{\eta}} \productofmarginals^k(\shiftedkpairsset{A}), \forall \eta \leq \eta_0.
    \end{equation}
    \noindent Then, 
    \begin{equation}
        \liminf_{\eta \downarrow 0} \eta \log \pi_\eta^k(A) \geq \liminf_{\eta \downarrow 0} \eta \log \productofmarginals^k(\shiftedkpairsset{A}) - \delta_+ = - \delta_+,
    \end{equation}
    \noindent and since this holds for all $\delta_+ > \delta'$, we have $\liminf_{\eta \downarrow 0} \eta \log \pi_\eta^k(A) \geq \delta'$.
\end{proof}

\begin{proposition}[cf. Proposition 3.2 in \cite{bernton2022entropic}]\label{prop:bernton32mod}
    Let $\pi$ be a cluster point of $\{ \pi_\eta \}_{\eta > 0}$.
    Then, $\operatorname{spt} \pi$ is cyclically monotone.
\end{proposition}

\begin{proof}
    Assume the contrary, i.e., $\exists \, (x_i, y_i)_{i=1}^k \subseteq \operatorname{spt} \pi$ such that $\sum_{i=1}^k c(x_i, y_i) - c(x_i, y_{i+1}) > 0$.
    By continuity of $c$, there exists a $\delta > 0$ and open neighborhoods $G_i \ni (x_i, y_i)$ such that, %
    \begin{equation}\label{eq:bernton32mod-1}
        \sum_{i=1}^k c(x_i, y_i) - c(x_i, y_{i+1}) > \delta, \quad\text{in} \ G \coloneq \prod_{i=1}^k G_i.
    \end{equation}
    Since $(x_i, y_i) \in \spt \pi$, we have $\pi(G_i) > 0$, so $\liminf_{\eta \downarrow 0} \pi_\eta^k(G) \geq \pi^k(G) > 0$.
    But \eqref{eq:bernton32mod-1} gives that $G \subseteq A_k(\delta, \infty)$, so $\limsup_{\eta \downarrow 0} \eta \log \pi_\eta^k(G) \leq -\delta$ by Lemma~\ref{lem:bernton31mod}, which can only hold if $\log \pi_\eta^k(G) \to -\infty$, and subsequently $\pi_\eta^k(G) \to 0$, a contradiction.
\end{proof}

\noindent Since the set of couplings $\Pi(\mu, \nu)$ is compact in the weak topology, the sequence $\{ \pi_\eta \}_{\eta > 0}$ is guaranteed to have a cluster point. 
Thus, by the same compactness, if there is a unique cyclically monotone $\pi \in \Pi(\mu, \nu)$, we have that $\pi_\eta \to \pi$ weakly. %
In the sequel, for simplicity of presentation and notation, we assume that $\pi_\eta \to \pi$ for some cluster point $\pi$.
This must not hold in general, but the results will be true along any convergent subsequence of $\{ \pi_\eta \}_{\eta > 0}$. In what follows, we use the notation that for a metric space $(\calX, d)$,  $B_\calX(x, r) \coloneq \{y \in x: d(x, y) < \varepsilon\}$ denotes the open ball of radius $r > 0$ around $x \in X$. %

\begin{lemma}[cf. Lemma 4.1 in \cite{bernton2022entropic}]\label{lem:bernton41mod}
    Let $(x,y) \in \calX \times \calY$.
    Then, for any $\delta < I(x, y)$ (given in \eqref{eq:rate-function-candidate-theorem}), there exists $r > 0$ such that %
    \begin{equation}
        \limsup_{\eta \downarrow 0} \eta \log \pi_\eta\big(B_{\calX \times \calY}((x,y), r)\big) \leq -\delta, %
    \end{equation}
    where $B_{\calX \times \calY}((x,y), r)$ denotes an open ball of radius $r$ around $(x, y)$ in $\calX \times \calY$.
\end{lemma}

\begin{proof}
    For $\delta \leq 0$, the result holds trivially, so we may restrict our attention to $\delta \geq 0$.
    If $\delta < I(x,y)$, there exists $(x_i, y_i)_{i=2}^k \subseteq \Gamma$ such that $\delta_0 \coloneq \sum_{i=1}^k c(x_i, y_i) - c(x_i, y_{i+1}) > \delta$, where $(x_1, y_1) = (x, y)$.
    Choose $r > 0$ small enough so that %
    \begin{equation}
        \sum_{i=1}^k c(\tilde{x}_i, \tilde{y}_i) - c(\tilde{x}_i, \tilde{y}_{i+1}) \geq \delta, \ \ \text{for all} \ (\tilde{x}_i, \tilde{y}_i)_{i=1}^k \in \prod_{i=1}^k B_{\calX \times \calY}((x_i,y_i), r).
    \end{equation}
    Then $\prod_{i=1}^k B_{\calX \times \calY}((x_i,y_i), r) \subseteq A_k(\delta, \infty)$.
    By Lemma \ref{lem:bernton31mod},
    \begin{equation}\label{eq:bernton41mod-1}
        \limsup_{\eta \downarrow 0} \eta \log \pi_\eta^k\big(\prod_{i=1}^k B_{\calX \times \calY}((x_i,y_i), r)\big) \leq -\delta,%
    \end{equation}
    We now use that the left-hand side of the previous display is 
    \begin{equation}\label{eq:bernton41mod-15}
    \begin{split}
        &\limsup_{\eta \downarrow 0} \eta \log \pi_\eta^k\big(\prod_{i=1}^k B_{\calX \times \calY}((x_i,y_i), r)\big) \\ 
        & \quad = \limsup_{\eta \downarrow 0} \left( \eta  \log \pi_\eta\big(B_{\calX \times \calY}((x,y), r)\big) + \eta \log \pi_\eta^{k-1}\big(\prod_{i=2}^{k} B_{\calX \times \calY}((x_i,y_i), r)\big) \right).
    \end{split}
    \end{equation}
    The second term inside the parenthesis has a limit $0$, %
    since for all $i$, by the Portmanteau theorem, 
    \begin{equation*}
    \begin{split}
        1 &\geq \limsup_{\eta \downarrow 0} \pi_\eta(B_{\calX \times \calY}((x_i,y_i), r)) \\ 
        & \geq \liminf_{\eta \downarrow 0} \pi_\eta(B_{\calX \times \calY}((x_i,y_i), r)) \\
        &\geq \pi(B_{\calX \times \calY}((x_i,y_i), r)) \\
        &> 0, %
    \end{split}
    \end{equation*}
    where we use that $(x_i, y_i) \in \spt \pi$. It follows that \eqref{eq:bernton41mod-1} is precisely
    \begin{equation}\label{eq:bernton41mod-2}
        \limsup_{\eta \downarrow 0} \eta \log \pi_\eta(B_{\calX \times \calY}((x,y), r)) \leq -\delta,%
    \end{equation}
   which completes the proof.
\end{proof}

Lemma \ref{lem:bernton41mod} gives the large deviation upper bound for compact sets, as stated next.

\begin{corollary}[cf. Corollary 4.3 in \cite{bernton2022entropic}]\label{cor:bernton43mod}
    For any compact set $K \subseteq \calX \times \calY$, 
    \begin{equation}
        \limsup_{\eta \downarrow 0} \eta \log \pi_\eta(K) \leq -\inf_{(x,y) \in K} I(x,y).%
    \end{equation}
\end{corollary}

\begin{proof}

    For any (finite) $\delta < \inf_{(x,y) \in K} I(x,y)$, %
    and each $(x, y) \in K$, %
    by Lemma~\ref{lem:bernton41mod}, there is an $r_{xy} > 0$ such that $\limsup_{\eta \downarrow 0} \eta \log\pi_\eta(B_{\calX \times \calY}((x,y), r_{xy}) < -\delta$.
    By compactness of $K$, choose a finite subcover for $K$ of such balls $\{B_j\}_{j=1}^{N}$, where $B_j = B_{\calX \times \calY}((x_j,y_j), r_{x_j y_j})$.
    Then, 
    \begin{equation}
    \begin{split}
        \limsup_{\eta \downarrow 0} \eta \log \pi_\eta(K) &\leq \limsup_{\eta \downarrow 0} \eta \log \pi_\eta(\cup_{j=1}^{N} B_j) \\
        &\leq \limsup_{\eta \downarrow 0} \eta \log \sum_{j=1}^{N} \pi_\eta(B_j) \\
        &\leq \limsup_{\eta \downarrow 0} \eta \log(N \max_j \pi_\eta(B_j)) \\
        &= \max_j \limsup_{\eta \downarrow 0} \eta \log \pi_\eta(B_j) \\
        &\leq -\delta.
    \end{split}
    \end{equation}
    Since this inequality holds for any $\delta < \inf_{(x,y) \in K} I(x, y)$, the claim follows.
\end{proof}

This completes the proof of the large deviation upper bound. 
The statements and proofs of the two following results, Proposition \ref{prop:bernton45mod} and Corollary \ref{cor:bernton47mod} are identical to \cite{bernton2022entropic}. We therefore omit them here and refer the reader to \cite{bernton2022entropic}.

\begin{proposition}[Proposition 4.5 in \cite{bernton2022entropic}]\label{prop:bernton45mod}
    ``Let Assumption \ref{ass:bernton44mod} hold.
    Then, $I$, defined in \eqref{eq:rate-function-candidate-theorem}, is also given by $I = c - (-\psi \oplus \psi^c)$ on $\calX_0 \times \calY_0$, for any Kantorovich potential $\psi$.
    In particular, $I < \infty$ on $\calX_0 \times \calY_0$.
    If $(x,y), (x',y') \in \calX_0 \times \calY_0$ are such that $(x',y), (x,y') \in \Gamma$, then $I(x,y)+I(x',y')=c(x,y)+c(x',y')-c(x,y')-c(x',y)$.''
\end{proposition}

\begin{corollary}[Corollary 4.7 in \cite{bernton2022entropic}]\label{cor:bernton47mod}
    ``Let Assumption \ref{ass:bernton44mod} hold.
    For any open set $G \subseteq \calX_0 \times \calY_0$,
    \begin{equation}
        \liminf_{\eta \downarrow 0} \eta \log \pi_\eta(G) \geq -\inf_{(x,y) \in G} I(x,y).\text{''}
    \end{equation}
\end{corollary}

Corollaries \ref{cor:bernton43mod} and \ref{cor:bernton47mod} give the large deviation upper and lower bounds in Theorem \ref{thm:weak-type-LDP}, respectively.
Proposition \ref{prop:bernton45mod} gives the alternate form of the rate function as stated in Theorem~\ref{thm:weak-type-LDP}, and combined these results thus prove Theorem \ref{thm:weak-type-LDP}. Lastly, %
Corollary \ref{cor:katocor31mod}, which gives the full LDP when we consider compact supports, can be proven as follows, as alluded to in \cite{Kato24}.

\begin{proof}[Proof of Corollary \ref{cor:katocor31mod}]
    If $\spt \mu, \spt \nu$ are compact, then $\restr{\proj_{\calX_0}}{\spt \mu \times \spt \nu}$ is a closed map, i.e. $\proj_{\calX_0}(F)$ is closed for any closed $F \subseteq \spt \mu \times \spt \nu$. %
    To see this, assume that $x \in \calX$ is a limit point of $\proj_{\calX_0}(F)$.
    Then there exists a sequence $\{(x_i, y_i)\}_{i = 1}^\infty \subseteq F$, with $\{x_i\}_{i = 1}^\infty$ converging to $x$. %
    By the compactness of $F$, %
    the sequence has a subsequence converging to $(x, y) \in F$ for some $y$,
    which implies $x \in \proj_{\calX_0}(F)$.
    Thus $\proj_{\calX_0}(F)$ is closed.
    This further implies that $\calX_0 = \proj_{\calX_0}(\Gamma)$ is closed, and thus $\calX_0 = \close{\calX_0} = \spt \mu$ (and similarly $\calY_0 = \spt \nu$).
    Also, under the assumption, any closed set in $\spt \mu \times \spt \nu$ is compact.
    Thus, from Theorem~\ref{thm:weak-type-LDP}, a full LDP with the same rate function holds on $\spt \mu \times \spt \nu$.
    By an easy argument based on that $\pi_\eta(A) = 0$ for any $A \in \calX \times \calY \setminus (\spt \mu \times \spt \nu)$ and all $\eta > 0$ (or the contraction principle \cite[Theorem~4.2.1]{Dembo98} applied to the inclusion map $\spt \mu \times \spt \nu \subseteq \calX \times \calY$), the LDP holds on $\calX \times \calY$ with the extended rate function in \eqref{eq:rate-function-extended-cor}.
\end{proof}

\section{Uniform convergence for reflected Brownian motion}\label{sec:small-noise-reflected-browian-motion}

In this section, we demonstrate the usefulness of the results of Section \ref{sec:ldp-uniform-convergent-cost}, by applying them to a specific choice of dynamics that has appeared in the Schr\"{o}dinger bridge literature, but that is not covered by the results of \cite{bernton2022entropic}, namely \textit{reflected Brownian Schr\"{o}dinger bridges} on bounded convex domains.
We start with a brief introduction to such dynamics; throughout the section we are working on some probability space $(\Omega, \calF, \bP)$.

\subsubsection*{Reflected SDEs/SBs}\label{sec:reflected-SB}

A useful modification of the reference dynamics \eqref{eq:small-scale-diffusion-drift-f} is given by adding reflection at the boundary of some  domain $D \subseteq \bR^d$, thus constraining $\{X_t\}$ to $\close{D}$ by reflecting (bouncing back) at the boundary $\partial D$.
This is often expressed as
\begin{equation}\label{eq:small-scale-reflected-diffusion-drift-f}
\begin{split}
    d&X^\eta_t = f(t, X^\eta_t) dt + \sqrt{\eta} dW_t + n(X_t) d\Lambda_t, \quad t \geq 0, \\
    &X^\eta_0 = X_0 \sim \mu,
\end{split}
\end{equation}
where $W$ is a standard Brownian motion (BM) under $\bP$, $n(x)$ is the inward directed normal (we do not consider oblique reflections) at $x \in \partial D$ and $0$ in $\interior{D}$, $\Lambda$ is a local time satisfying $\Lambda_t = \int_0^t \mathbbm{1}_{X_s \in \partial D} d\Lambda_s$, $X_t \in \close{D} \, \, \forall t \in [0,1]$, and $X_0 \sim \mu$ with $\spt \mu \in \close{D}$.
Alternatively, $X_t = \Gamma Y_t$ for $Y_t$ satisfying
\begin{equation}\label{eq:small-scale-reflected-diffusion-drift-f-Skorokhod}
\begin{split}
    d&Y^\eta_t = f(t, \Gamma Y^\eta_t) dt + \sqrt{\eta} dW_t,%
\end{split}
\end{equation}
where $\Gamma: \calC_1 \to \calC_{1}$ is the \emph{Skorokhod map} for $D$, defined by $f \mapsto g$ in the \emph{Skorokhod problem}: for $f \in \calC_1$, find $(g, l)$, with $g \in \calC_1, l: [0,1] \to \bR$, such that
\begin{itemize}
    \item $g = f + \int_0^{\cdot} n(g(s)) dl(s)$ with $n$ being the inwards normal mentioned above, %
    \item %
    $g(t) \in \close{D} \, \forall t \geq 0$,
    \item %
    $l = \int_0^{\cdot} \mathbbm{1}_{g(s) \in \partial D} dl(s)$.
\end{itemize}
Existence and uniqueness of a solution to this problem hold under fairly mild conditions on $D$ \cite[Theorem 1.1]{lions1984stochastic}.
Further, $\Gamma$ is $\frac{1}{2}$-H\"{o}lder under these conditions.
Under stronger assumptions, like $D$ being polygonal, $\Gamma$ can be shown to be Lipschitz \cite{dupuis1999convex, dupuis1993sdes}. %
Similar to standard (i.e., non-reflected) SDEs, the solution is a strong Markov process under well-posedness, that is also time-homogeneous if $f$ does not depend on $t$.

For $f = 0$, we obtain the ($\eta$-scaled) \emph{reflected Brownian motion} (RBM), simply given by the $X^\eta = \Gamma(X_0 + \sqrt{\eta} W)$ (a strong solution).
This type of reference process is the primary focus in this section.
We denote the transition density of the reflected process $X^\eta$ by $p^r_\eta(t, x, y)$, and by $p_\eta(t, x, y)$ for the unreflected $X_0 + \sqrt{\eta} W$.
Similar to $p_\eta$, $p^r_\eta$ obeys the time scale-invariance relation: $p^r_\eta(t, x, y) = p^r_1(\eta t, x, y)$.
Further, $p^r_\eta$ can be characterized as a \emph{Neumann heat kernel}, solving the heat equation with Neumann boundary conditions, starting from a point mass at $x$: letting $\nabla_y, \Delta_y$ denote the gradient and Laplacian in $y$, respectively,%
\begin{equation}\label{eq:Fokker-Planck-reflected}
\begin{split}
    \frac{\partial}{\partial t} p^r_\eta (t, x, y) &= \frac{\eta}{2} \Delta_y p^r_\eta (t, x, y), \quad t > 0, \, y \in \close{D} \\
    n(x) \cdot \nabla_y p(t, x, y) &= 0, \quad y \in \partial D,  \\
    p(t, x, y) &\to \delta_x, \ \text{as} \ t \downarrow 0.
\end{split}
\end{equation}
For more information on reflected SDEs, see \cite{pilipenko2014introduction}.

A \emph{reflected SB} refers to a SB $\pi_\eta \coloneq \pi^{R_\eta}$, where $R_\eta$ is the path measure of the reflected SDE \eqref{eq:small-scale-reflected-diffusion-drift-f}.
As before, we are interested in the entire collection $\{\pi_\eta\}_{\eta > 0}$. %
Reflected SBs were introduced and computed in a low-dimensional setting in \cite{caluya2021reflected}.
In \cite{deng2024reflected} they were incorporated in high-dimensional generative modeling using forward-backward SDE theory \cite{chen2021likelihood}.

Motivated by their use in applications within data science, we want to apply the results of Section \ref{sec:ldp-uniform-convergent-cost} to reflected Brownian SBs, on convex domains $D$.
Thus, we need to establish that $\{c_\eta\}_{\eta > 0}$ converges uniformly (to $c(x,y) = \frac{1}{2}|x-y|^2$), where $c_\eta$ is given by (see Section \ref{sec:ceta-explicit-form})
\begin{equation}\label{eq:c_eta-UB-LB-section}
    c_\eta(x,y) = -\eta \log p^r_\eta(1, x, y).
\end{equation}
As explained in Section \ref{sec:ceta-explicit-form}, using this form relies on the following assumption.

\begin{assumption}
    The differential entropy of $\nu$, $\calH(\nu) \coloneq - \int_{\bR^d} \log \frac{d\nu}{d\lambda} \, d\nu$ is finite.
\end{assumption}

In order to establish the necessary convergence criterion, we will study the transition density $p^r_\eta$, specifically its asymptotic upper and lower bounds as $\eta \downarrow 0$. %

\subsection{Examples with explicit transition densities}

For $D = [0, \infty)$ with reflection at $0$, the Skorokhod map is given by $\Gamma f(t) = f(t) -(\inf_{s \leq t} f(s) \land 0)$.
One can show that the transition density of $\Gamma W$ is 
\begin{equation}\label{eq:transition_density_halfinfinite}
    p_+(t, x, y) \coloneq p(t, x, y) + p(t, -x, y),
\end{equation}
For $D=[0,1]$, an explicit expression for the Skorokhod map is also available, see \cite{kruk2007explicit}.
The transition density of $\Gamma W$ is then given by
\begin{equation}\label{eq:transition_density_compact}
    \bP(\Gamma W_t \in dy \mid W_0 = x) = dy \sum_{n = -\infty}^\infty p_+(t, x + 2n, y)
\end{equation}
See \cite[p. 97]{KS08}  for a reference to the formulas in \eqref{eq:transition_density_halfinfinite} and \eqref{eq:transition_density_compact}.
A physical interpretation of them is that unrestrained rays starting at $x$ and ending in certain points will instead end at $y$ when reflected, i.e. under the Skorokhod map, see also 
\cite{lou2023reflected, harrison1981reflected}.
Using the formulas, it is straightforward to show that $c_\eta$ converges uniformly to $c(x,y) = \frac{1}{2} |x - y|^2$ in both cases, see Figure \ref{fig:reflected_01_c_eta_convergence} for the latter case.
By the independence of dimensions, these results also extend to $[0, 1]^d$ and $[0, \infty)^d$.
Further, we get an explicit expression for the Doob $h$-transform necessary to investigate the bridge processes of RBM.

Polygonal domains $D$ are treated in \cite{dupuis1993sdes} and \cite{harrison1981reflected}.
It is shown in \cite{dupuis1993sdes} that $\Gamma$ is then Lipschitz.
This could simplify some of our arguments, 
but we refrain from using this assumption on the domain.

\begin{figure}
    \centering
    \includegraphics[width=0.5\linewidth]{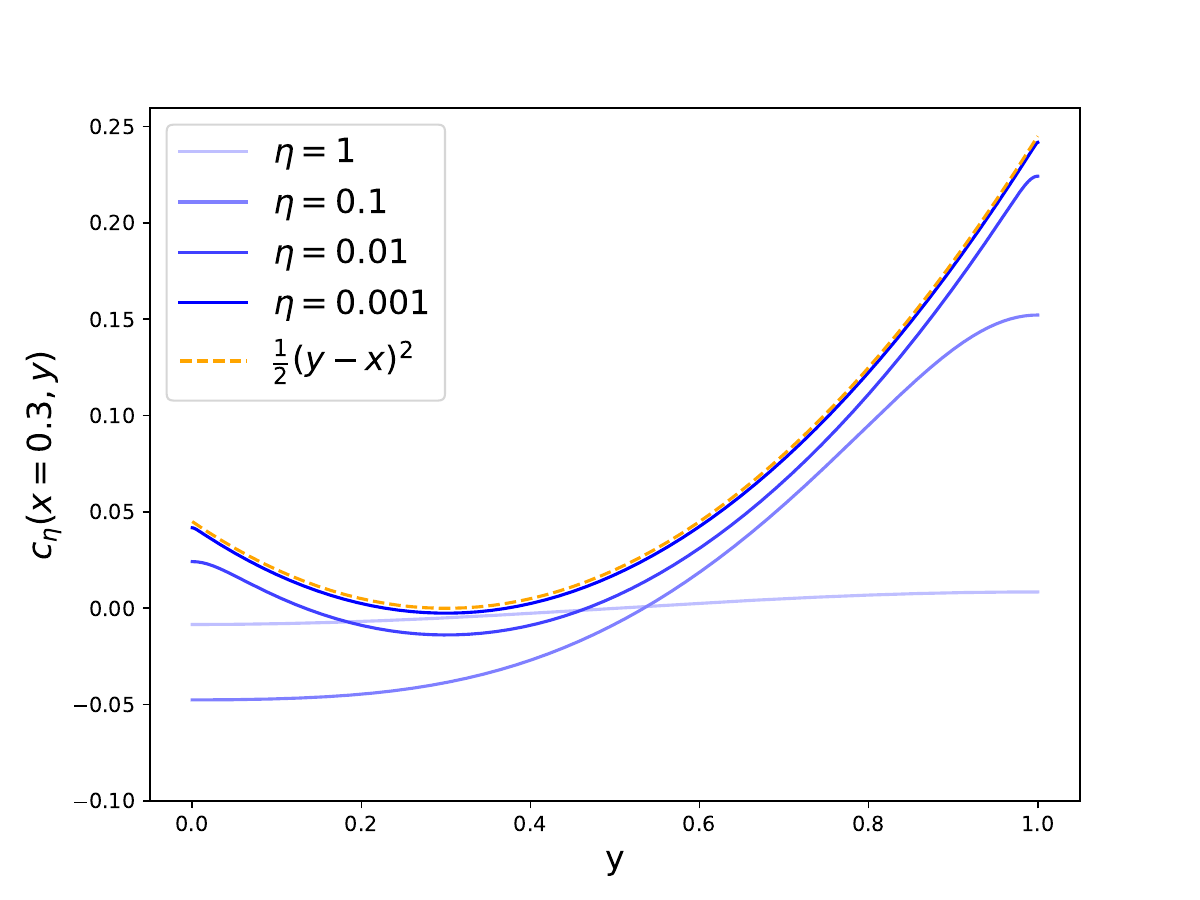}
    \caption{Uniform convergence of $c_\eta$ on $D = [0, 1]$, computed using \eqref{eq:transition_density_compact} on the slice $x=0.3$.}
    \label{fig:reflected_01_c_eta_convergence}
\end{figure}

\subsection{The general bounded convex case}

We now prove convergence in a more general setting. %
Assume that $D$ is an open, bounded, and convex set. %
This then satisfies the ``extension property'', see \cite[p. 47]{davies1989heat}, and the mild conditions in \cite{lions1984stochastic}, meaning $\Gamma$ is continuous.
Further $\interior{D}$ is connected, so Assumption \ref{ass:bernton44mod} holds.
The goal is the following general theorem.
\begin{theorem}\label{thm:reflected-ceta-uniform-convergence}
    With $D$ open, bounded, and convex, and $c_\eta(x,y) \coloneq -\eta \log p^r_\eta(1, x, y)$,
    $\{c_\eta\}_{\eta > 0}$ converges uniformly to $c(x,y) = \frac{1}{2}|x-y|^2$ as $\eta \downarrow 0$.
\end{theorem}
From results on the heat equation, we have the following \textbf{upper bound} on the transition function $p^r_\eta$; see \cite[p.\ 90]{davies1989heat}.
\begin{proposition}[Theorem 3.2.9 in \cite{davies1989heat}]\label{prop:davies329}
    If $D \subseteq \bR^d$ is a bounded region with the extension property, then the Neumann heat kernel $p^r_\eta(1, x, y)$ satisfies, for any $\delta > 0$, and some constant $c_{\delta, D}$,
    \begin{equation}
        0 \leq p^r_\eta(1, x, y) \leq c_{\delta, D}%
        \exp{\left\{-\frac{|x-y|^2}{2(1 + \delta)\eta}\right\}}.
    \end{equation}
\end{proposition}
To prove the uniform convergence of \eqref{eq:c_eta-UB-LB-section} needed to apply Theorem \ref{thm:weak-type-LDP}, we require the corresponding \textbf{lower bound}. 
Establishing such a lower bound is the content of the remainder of this section. 
From \cite{lions1984stochastic}, we have that for the type of domain $D$ under consideration, $\Gamma$ is continuous (in fact with H\"{o}lder coefficient $1/2$) on compact sets of $\mathcal{C}_1$.
We will use this to show a lower bound for $p^r_\eta$. 
We start by obtaining a lower bound on the following family of compactly embedded convex sets $D^{-\ve}$ in $D$,
\begin{equation}
    D^{-\varepsilon} \coloneq \{x \in D: |x - b| > \varepsilon, \forall b \in \partial D \} \quad \varepsilon > 0.
\end{equation}

\begin{lemma}
    If $D \subseteq \bR^d$ is an open convex set, then $D^{-\varepsilon}$ is convex.
\end{lemma}

\begin{proof}
    Assume the contrary, i.e. that for some $x, y \in D^{-\varepsilon}, b \in \partial D, t_a \in [0, 1]$, it holds that $a \coloneq t_a x + (1 - t_a)y$ satisfies
    $|a - b| \leq \varepsilon$, so that $a \notin D^{-\varepsilon}$.
    Then one has that $x_\varepsilon \coloneq x + (b - a)$ and $y_\varepsilon \coloneq y + (b - a)$ are contained in $D$, since $|x_\varepsilon - x| = |y_\varepsilon - y| = |b - a| \leq \varepsilon$ and $x, y \in D^{-\varepsilon}$.
    By convexity of $D$, it then holds that $a_\varepsilon \coloneq t_a x_\varepsilon + (1-t_a) y_\varepsilon \in D$.
    But $a_\varepsilon = t_a (x + (b - a)) + (1-t_a) (y + (b - a)) = b$, contradicting that $b \in \partial D$, since $D$ is a domain (i.e. open, connected).
\end{proof}

Let us review some notation. For $t>0$, let $\calC_t \coloneq C([0,t]: \bR^d)$ denote the space of continuous functions $[0,t] \to \bR^d$, endowed with the sup-norm: $d_{\calC_t}(f,g) \coloneq \sup_{s \in [0,t]} |f(s) - g(s)|$.
Consider an $\eta$-scaled Brownian motion starting in $x \in D$, $\{W^{\eta, x}\} = \{x + \sqrt{\eta} W_t\}$, and its reflected version $\Gamma W^{\eta, x}$. %
Define also the \emph{Brownian bridge} (BB) $B^{\eta, xy}$ as a process with the law of $W^{\eta, x}$ conditioned on ending in $y$ at time $t=1$, and let $B^\eta \coloneq B^{\eta, 00}$.
With $\sigma^{xy} \in \calC_1$ %
defined by $\sigma^{xy}(t) \coloneq x + t(y-x)$, note that $B^{\eta, xy} = B^{\eta} + \sigma^{xy}$.
Also, for general ending times $t > 0$, define $\sigma^{xy}_{0t} \in \calC_t$ by $\sigma^{xy}_{0t}(s) \coloneq \sigma^{xy}(s/t)$, and similarly $B^{\eta, xy}_{0t}$ and $B^{\eta}_{0t}$ as Brownian bridges ending at time $t$ (with the analogous property $B^{\eta, xy}_{0t} = B^{\eta}_{0t} + \sigma^{xy}_{0t}$).

The following Lemma is formally the result we want, but restricted to $D^{-\varepsilon}$.
The full result requires more work.

\begin{lemma}\label{lem:uniform-lb-D-eps}
    For any $\varepsilon > 0$, there exists $\eta_0 = \eta _0 (\varepsilon) > 0$ such that for any $\eta \leq \eta_0$, $x,y \in D^{-\varepsilon}$, we have
    \begin{equation}
        p^r_\eta(1, x, y) \geq \frac{1}{2} p_\eta(1, x, y).%
    \end{equation}
\end{lemma}

\begin{remark}\label{rem:uniform-lb-D-eps}
    Note the key property that $\eta_0$ does not depend on $(x, y)$, only $\varepsilon$.
    Also, note that $t=1$ can be generalized to $t\in(0,1]$ by rescaling time, $\eta \to \eta t$, using the time-scale invariance of RBM and BM.
    Thus, it follows from Lemma \ref{lem:uniform-lb-D-eps} that for all $t \in (0,1]$, $p^r_\eta(t, x, y) \geq \frac{1}{2} p_\eta(t, x, y)$.
\end{remark}

\begin{proof}[Proof of Lemma \ref{lem:uniform-lb-D-eps}]
    Define the stopping time $\tau_D \coloneq \inf \{t \geq 0: W_t \notin D\}$.
    We get,
    \begin{equation}
    \begin{split}
        \bP(\Gamma W^{\eta, x}_1 \in dy) &\geq \bP(\Gamma W^{\eta, x}_1 \in dy, \tau_D > 1) \\
        &= \bP(W^{\eta, x}_1 \in dy, \tau_D > 1) \\
        &\geq \bP(W^{\eta, x}_1 \in dy, \restr{W^{\eta, x}}{[0,1]} \in B_{\calC_1}(\sigma^{xy}, \varepsilon)) \\
        &= \bP(W^{\eta, x}_1 \in dy) \bP\left(\restr{W^{\eta, x}}{[0,1]} \in B_{\calC_1}(\sigma^{xy}, \varepsilon) \mid W^{\eta, x}_1 = y\right) \\
        &= \bP(W^{\eta, x}_1 \in dy) \bP\left(\sup_{t \in [0,1]} |W^{\eta, x}_t - \sigma^{xy}_t| < \varepsilon \mid W^{\eta, x}_1 = y\right) \\
        \text{\{Brownian bridge\}}&= \bP(W^{\eta, x}_1 \in dy) \bP\left(\sup_{t \in [0,1]} |B^{\eta}_t| < \varepsilon\right).
    \end{split}
    \end{equation}
    The second factor converges to $1$ as $\eta \downarrow 0$, since the supremum norm of any scaled path $\sqrt{\eta} f \in \calC_1$ goes to zero with $\eta$.
    The estimate follows for small enough $\eta$.
\end{proof}

Next, we use Lemma \ref{lem:uniform-lb-D-eps} to obtain a lower bound on $p^r _\eta (1,x,y)$ that holds uniformly for $x,y \in \close{D}$. 
To ease notation, we take %
$\operatorname{diam}(D) = \sup_{x,y \in D} |x-y|$ to denote the (finite by assumption) diameter of the set $D$.%

\begin{proposition}\label{prop:uniform-lb}
    Take $\varepsilon \in (0,\frac{1}{2}\wedge \operatorname{diam}(D))$. Then, there exists $\alpha_D = \alpha_D(\varepsilon)> 0, \ \eta_0 = \eta_0 (\varepsilon) > 0$, and $\beta_D >0 $, such that for any $x, y \in \close{D}$ and $\eta \leq \eta_0$,
    \begin{equation}
        p^r_\eta(1, x, y) \geq \alpha_D \exp{-\left\{\frac{|x - y|^2 + \beta_D \varepsilon}{2 \eta (1 - \varepsilon)}\right\}}.
    \end{equation}
\end{proposition}

\begin{proof}%
    We start by proving a lower bound for $p ^r _\eta (1,x,y)$ when $y \in D^{-\varepsilon}$ and $x \in \close{D}$. 
    Take $s = 1 - \frac{\varepsilon}{1 \lor 3\operatorname{diam}(D)}$ %
    so that    
    $(1-s)|x - y| < \varepsilon / 3$ (note also that $s \geq \frac{1}{2}$, which will be used later). 
    Letting $z \coloneq x + s(y - x)$, this means that $z \in D^{-2\varepsilon / 3}$, and $B_{\bR^d}(z, \varepsilon / 3) \subseteq D^{-\varepsilon / 3}$.
    $z$ will be used as a ``bridge point'' between $x$ and $y$.
    By a Chapman-Kolmogorov equation restricted to a ball, we get

    \begin{equation}\label{eq:uniform-lb-eq-1}
    \begin{split}
        \bP(\Gamma W^{\eta, x}_1 \in dy) &\geq \bP(\Gamma W^{\eta, x}_1 \in dy, \Gamma W^{\eta, x}_s \in B_{\bR^d}(z, \varepsilon / 3)) \\
        &= \int_{z' \in B_{\bR^d}(z, \varepsilon / 3)} \bP(\Gamma W^{\eta, x}_1 \in dy \mid W_s = z') \bP(\Gamma W^{\eta, x}_s = dz') \\
    \end{split}
    \end{equation}%
    By the time-homogeneous Markov property of RBM, the conditional probability in the last line above is
    \begin{equation}\label{eq:uniform-lb-eq-2}
        \bP(\Gamma W^{\eta, x}_1 \in dy \mid W_s = z') = \bP(\Gamma W^{\eta, z'}_{1-s} \in dy),
    \end{equation}
    \noindent and for all $\eta \leq \eta_0(\varepsilon / 3)$, according to Lemma \ref{lem:uniform-lb-D-eps}, this is bounded below by $\frac{1}{2} p_\eta((1-s), z', y) dy = \frac{1}{2} (2\pi \eta (1-s))^{-d/2} \exp{-\frac{|z' - y|^2}{2\eta(1-s)}} dy$ (see Remark \ref{rem:uniform-lb-D-eps}). 
    Inside $B_{\bR^d}(z, \varepsilon / 3)$, this is in turn bounded by $\frac{1}{2} (2\pi \eta (1-s))^{-d/2} \exp{-\frac{(|z - y| + \varepsilon / 3)^2}{2\eta(1-s)}} dy$, which does not depend on the integration variable $z'$.
    Continuing on \eqref{eq:uniform-lb-eq-1} gives
    \begin{equation}\label{eq:uniform-lb-eq-3}
    \begin{split}
        &\bP(\Gamma W^{\eta, x}_1 \in dy) \\ 
        &\geq \frac{1}{2} \frac{1}{(2\pi \eta (1-s))^{d/2}} \exp{-\frac{(|z - y| + \varepsilon / 3)^2}{2\eta(1-s)}} \bP\Big(\Gamma W^{\eta, x}_s \in B_{\bR^d}(z, \varepsilon / 3)\Big) dy.
    \end{split}
    \end{equation}
    The probability appearing last in \eqref{eq:uniform-lb-eq-3} can be bounded via, as in the proof of Lemma \ref{lem:uniform-lb-D-eps}, taking open balls in $\calC_s$.
    However, one cannot assume that $\Gamma W^{\eta, x} = W^{\eta, x}$ within these balls, since $x$ may be arbitrarily close to the boundary.
    Instead, we will use that %
    $\Gamma$ is uniformly continuous on (relatively) compact sets of $\calC_s$.
    For $t \in (0, 1]$, let $A^t_{\delta_0}$ be the relatively compact set %
    in $\calC_t$ given by 
    \begin{equation}
        A^t_{\delta_0} \coloneq \{f \in \calC_t: 
        f(0) \in \close{D}, 
        m^t(f, \delta) \leq 2 g(\delta) \, \forall \delta \leq \delta_0\},
    \end{equation}
    for some $\delta_0 > 0$ where 
    \(m^t(f, \delta) \coloneq \max_{\substack{%
    t_1, t_2 \in [0, t], |t_1-t_2| \leq \delta}} |f(t_1) - f(t_2)|\)
    denotes the \emph{modulus of continuity} on $[0,t]$, and $g(\delta) = \sqrt{2 \delta \log(1 / \delta)}$, see \cite[p. 62, and p. 114]{KS08}.
    Then $\Gamma$ is uniformly continuous on $A^t_{\delta_0}$.
    Let $\varepsilon^\calC_{\delta_0}$ be small enough so that $\Gamma(B_{\calC_1}(f, \varepsilon^\calC_{\delta_0}) \cap A_{\delta_0}) \subseteq B_{\calC_1}(f, \varepsilon / 3)$ %
    for each $f \in A_{\delta_0}$. %
    Then,
    
    \begin{equation}\label{eq:uniform-lb-eq-4}
    \begin{split}
        &\bP\Big(\Gamma W^{\eta, x}_s \in B_{\bR^d}(z, \varepsilon / 3)\Big) \\
        &\geq \bP\Big(\restr{\Gamma W^{\eta, x}}{[0, s]} \in B_{\calC_s}(\sigma_{0s}^{xz}, \varepsilon / 3)\Big) \\
        &\geq \bP\Big(\restr{\Gamma W^{\eta, x}}{[0, s]} \in B_{\calC_s}(\sigma_{0s}^{xz}, \varepsilon / 3), \restr{W^{\eta, x}}{[0, s]} \in A^s_{\delta_0}\Big) \\
        &\geq \bP\Big(\restr{W^{\eta, x}}{[0, s]} \in B_{\calC_s}(\sigma_{0s}^{xz}, \varepsilon^\calC_{\delta_0}) \cap A^s_{\delta_0}\Big) \\
        &\geq \int_{B_{\bR^d}(z, \varepsilon^\calC_{\delta_0} / 2)} 
        \bP\left(W^{\eta, x}_s \in dz'\right)
        \bP\left(\restr{W^{\eta, x}}{[0, s]} \in B_{\calC_s}(\sigma_{0s}^{xz}, \varepsilon^\calC_{\delta_0}) \cap A^s_{\delta_0} \Bigm\vert W^{\eta, x}_s = z' \right). \\
    \end{split}
    \end{equation}
    The last conditional probability may be bounded below by a union bound.
    \begin{equation}\label{eq:uniform-lb-eq-5}
    \begin{split}
        &\bP\left(\restr{W^{\eta, x}}{[0, s]} \in B_{\calC_s}(\sigma_{0s}^{xz}, \varepsilon^\calC_{\delta_0}) \cap A^s_{\delta_0} \Bigm\vert W^{\eta, x}_s = z' \right) \\
        &\geq 1 - \bP\left(\restr{W^{\eta, x}}{[0, s]} \notin B_{\calC_s}(\sigma_{0s}^{xz}, \varepsilon^\calC_{\delta_0}) \Bigm\vert W^{\eta, x}_s = z'\right) -
        \bP\left(\restr{W^{\eta, x}}{[0, s]} \notin A^s_{\delta_0} \Bigm\vert W^{\eta, x}_s = z' \right) \\
    \end{split}
    \end{equation}
    Since the integral (over $z'$) in \eqref{eq:uniform-lb-eq-4} was restricted to $B_{\bR^d}(z, \varepsilon^\calC_{\delta_0} / 2)$, we have that $B_{\calC_s}(\sigma_{0s}^{xz'}, \varepsilon^\calC_{\delta_0} / 2) \subseteq B_{\calC_s}(\sigma_{0s}^{xz}, \varepsilon^\calC_{\delta_0})$ for any relevant $z'$.
    Then the first negative term of the right-hand side in \eqref{eq:uniform-lb-eq-5} is bounded by
    \begin{equation}\label{eq:uniform-lb-eq-6}
    \begin{split}
        &\bP\left(\restr{W^{\eta, x}}{[0, s]} \notin B_{\calC_s}(\sigma_{0s}^{xz}, \varepsilon^\calC_{\delta_0}) \Bigm\vert W^{\eta, x}_s = z'\right) \\
        &\leq \bP\left(\restr{W^{\eta, x}}{[0, s]} \notin B_{\calC_s}(\sigma_{0s}^{xz'}, \varepsilon^\calC_{\delta_0} / 2) \Bigm\vert W^{\eta, x}_s = z'\right) \\
        &= \bP\left(\sup_{t \in [0,s]} |W^{\eta, x}_t - \sigma_{0s}^{xz'}(t) | \geq \varepsilon^\calC_{\delta_0} / 2 \Bigm\vert W^{\eta, x}_s = z'\right) \\
        &= \bP\left(\sup_{t \in [0,s]} |W^{\eta, 0}_t| \geq \varepsilon^\calC_{\delta_0} / 2 \Bigm\vert W^{\eta, x}_s = 0\right),
    \end{split}
    \end{equation}
    which goes to $0$, again since the norm of any scaled path $\sqrt{\eta} f$ does. %
    To see that the second negative term of \eqref{eq:uniform-lb-eq-5} also goes to $0$, let $L_{\mathrm{max}} \coloneq 2 \operatorname{diam}(D)$ %
    and take $\delta_0 \coloneq e^{-1} \land L_{\mathrm{max}}^{-2}$. %
    We have chosen $s$ to be greater than $\frac{1}{2}$, %
    meaning that $|\sigma_{0s}^{xz'}(t_1) - \sigma_{0s}^{xz'}(t_2)| \leq g(|t_1 - t_2|)$, whenever $|t_1 - t_2| \leq \delta_0$.
    Then we get  
    \begin{equation}\label{eq:uniform-lb-eq-7}
    \begin{split}
        &\bP\left(\restr{W^{\eta, x}}{[0, s]} \notin A^s_{\delta_0} \Bigm\vert W^{\eta, x}_s = z' \right) \\
        &= \bP\left(B^{\eta, xz'}_{0s} \notin A^s_{\delta_0}\right) \\
        &= \bP\left(B^{\eta}_{0s} + \sigma_{0s}^{xz'} \notin A^s_{\delta_0}\right) \\
        &= \bP\left(\max_{\substack{t_1, t_2 \in [0, s] \\ |t_1-t_2| \leq \delta_0}} |(B^{\eta}_{0s} + \sigma_{0s}^{xz'})(t_1) - (B^{\eta}_{0s} + \sigma_{0s}^{xz'})(t_2)| > 2g(|t_1 - t_2|) \right) \\
        &\leq \bP\left(\max_{\substack{t_1, t_2 \in [0, s] \\ |t_1-t_2| \leq \delta_0}} |B_{0s}{\eta}(t_1) - B^{\eta}_{0s}(t_2)| + |\sigma_{0s}^{xz'}(t_1) - \sigma_{0s}^{xz'}(t_2)| > 2g(|t_1 - t_2|) \right) \\
        &\leq \bP\left(\max_{\substack{t_1, t_2 \in [0, s] \\ |t_1-t_2| \leq \delta_0}} |B_{0s}^{\eta}(t_1) - B_{0s}^{\eta}(t_2)| > g(|t_1 - t_2|) \right) \\ & \quad + \bP\left(\max_{\substack{t_1, t_2 \in [0, s] \\ |t_1-t_2| \leq \delta_0}} |\sigma_{0s}^{xz'}(t_1) - \sigma_{0s}^{xz'}(t_2)| > g(|t_1 - t_2|) \right) \\
    \end{split}
    \end{equation}
    The second term is zero.
    The first term goes to zero as $\eta \to 0$ since $g$ is (a multiple of) an exact modulus for a standard Brownian bridge and Brownian motion. %

    Returning to \eqref{eq:uniform-lb-eq-5}, we now have
    \begin{equation}\label{eq:uniform-lb-eq-8}
        \bP\left(\restr{W^{\eta, x}}{[0, s]} \in B_{\calC_s}(\sigma_{0s}^{xz}, \varepsilon^\calC_{\delta_0}) \cap A^s_{\delta_0} \Bigm\vert W^{\eta, x}_s = z' \right) \geq \frac{1}{2},
    \end{equation}
    for sufficiently small $\eta$.
    \eqref{eq:uniform-lb-eq-4} then gives
    \begin{equation}\label{eq:uniform-lb-eq-9}
    \begin{split}
        &\bP\Big(\Gamma W^{\eta, x}_s \in B_{\bR^d}(z, \varepsilon / 3)\Big) \\
        &\geq \frac{1}{2} \int_{B_{\bR^d}(z, \varepsilon^\calC_{\delta_0} / 2)} 
        \bP\left(W^{\eta, x}_s \in dz'\right) \\
        &\geq \frac{1}{2} \frac{1}{(2\pi \eta s)^\frac{d}{2}} \exp{-\frac{(|x - z| + \varepsilon^\calC_{\delta_0})^2}{2 \eta s}} \operatorname{vol}_{\bR^d}\Big(B_{\bR^d}\big(0, \frac{\varepsilon^\calC_{\delta_0}}{2}\big)\Big).
    \end{split}
    \end{equation}
    Inserting this into \eqref{eq:uniform-lb-eq-3} yields for an appropriate constant $\alpha$,
    \begin{equation}\label{eq:uniform-lb-eq-10}
    \begin{split}
        \bP(&\Gamma W^{\eta, x}_1 \in dy) \\
        \geq & \,\frac{1}{2} \frac{1}{(2\pi \eta (1-s))^{d/2}} \exp{-\frac{(|z - y| + \varepsilon / 3)^2}{2\eta(1-s)}} \\
        &\times \frac{1}{2} \frac{1}{(2\pi \eta s)^{d/2}} \exp{-\frac{(|x - z| + \varepsilon^\calC_{\delta_0})^2}{2 \eta s}} \operatorname{vol}_{\bR^d}\Big(B_{\bR^d}\big(0, \frac{\varepsilon^\calC_{\delta_0}}{2}\big)\Big) \\
        = & \,\alpha \eta^{-d} \exp{-\left\{\frac{(|z - y| + \varepsilon / 3)^2}{2\eta(1-s)} + \frac{(|x - z| + \varepsilon^\calC_{\delta_0})^2}{2 \eta s}\right\}}.
    \end{split}
    \end{equation}
    By only considering $\eta_0 \leq 1$, we may further drop the factor $\eta^{-d}$ from the final expression in \eqref{eq:uniform-lb-eq-10}.
    Introducing another appropriate constant $\beta$ in the exponent, and assuming w.l.o.g. that $\varepsilon^\calC_{\delta_0} \leq \varepsilon \leq 1 / 2$, we can further simplify the lower bound.%
    \begin{equation}\label{eq:uniform-lb-eq-11}
    \begin{split}
        \bP(\Gamma W^{\eta, x}_1 \in dy) &\geq \alpha {} \exp{-\left\{\frac{(|z - y| + \varepsilon / 3)^2}{2\eta(1-s)} + \frac{(|x - z| + \varepsilon^\calC_{\delta_0})^2}{2 \eta s}\right\}} \\
        &\geq \alpha {} \exp{-\left\{\frac{(\varepsilon)^2}{2\eta(1-s)} + \frac{(|x - z| + \varepsilon^\calC_{\delta_0})^2}{2 \eta s}\right\}} \\
        &\geq \alpha {} \exp{-\left\{\frac{(\varepsilon)^2}{2\eta\frac{\varepsilon}{1 \lor 3 \operatorname{diam}(D)}} + \frac{|x - z|^2 + 2 \operatorname{diam}(D) \varepsilon^\calC_{\delta_0} + (\varepsilon^\calC_{\delta_0})^2}{2 \eta (1 - \frac{\varepsilon}{1 \lor 3 \operatorname{diam}(D)})}\right\}} \\
        &= \alpha {} \exp{-\left\{\frac{(1 \lor 3 \operatorname{diam}(D))\varepsilon}{2\eta} + \frac{|x - z|^2 + 2 \operatorname{diam}(D) \varepsilon^\calC_{\delta_0} + (\varepsilon^\calC_{\delta_0})^2}{2 \eta (1 - \frac{\varepsilon}{1 \lor 3 \operatorname{diam}(D)})}\right\}} \\
        &\geq \alpha {} \exp{-\left\{\frac{(1 \lor 3 \operatorname{diam}(D))\varepsilon}{2\eta (1 - \varepsilon)} + \frac{|x - z|^2 + 2 \operatorname{diam}(D) \varepsilon^\calC_{\delta_0} + (\varepsilon^\calC_{\delta_0})^2}{2 \eta (1 - \varepsilon)}\right\}} \\
        &\geq \alpha \exp{-\left\{\frac{|x - z|^2 + \beta \varepsilon}{2 \eta (1 - \varepsilon)}\right\}} \\
        &\geq \alpha {} \exp{-\left\{\frac{|x - y|^2 + \beta \varepsilon}{2 \eta (1 - \varepsilon)}\right\}}
    \end{split}
    \end{equation}
    This gives the desired expression, although only when $x \in \close{D}, y \in D^{-\varepsilon}$.
    By symmetry of the transition function, it also holds when $x \in D^{-\varepsilon}, y \in \close{D}$.
    This symmetry may be seen by noting that a solution to the PDE \eqref{eq:Fokker-Planck-reflected}, also satisfies the Kolmogorov backward version of it, when switching $x$ and $y$. %

    For $x, y$ arbitrarily chosen in $\close{D}$, let $x^{\varepsilon'} \coloneq \operatorname{proj}_{\close{D^{-\varepsilon'}}}(x), y^{\varepsilon'} \coloneq \operatorname{proj}_{\close{D^{-\varepsilon'}}}(y)$ and set $z \coloneq \frac{1}{2} x^{\varepsilon'} + \frac{1}{2} y^{\varepsilon'} \in \close{D^{-\varepsilon'}} \subseteq D^{-\varepsilon' / 2}$.
    Here, $\varepsilon' \leq \varepsilon$ is chosen so that $\operatorname{dist}(x, \close{D^{-\varepsilon'}}) \leq \varepsilon$.
    Then, using a Chapman-Kolmogorov argument,
    \begin{equation}\label{eq:uniform-lb-eq-12}
    \begin{split}
        p^r_\eta(1, x, y) &\geq \int_{B_{\bR^d}(z, \varepsilon' / 2)} p^r_\eta(1/2, x, z') p^r_\eta(1/2, z', y) dz' \\
        &\geq \int_{B_{\bR^d}(z, \varepsilon' / 2)} \alpha {} \exp{-\left\{\frac{|x - z'|^2 + \beta \varepsilon'}{2 \eta (1/2) (1 - \varepsilon')}\right\}} \times \alpha {} \exp{-\left\{\frac{|z' - y|^2 + \beta \varepsilon'}{2 \eta (1/2) (1 - \varepsilon')}\right\}} dz' \\
        &= \alpha^2 {} \int_{B_{\bR^d}(z, \varepsilon' / 2)} \exp{-\left\{\frac{|x - z'|^2 + |z' - y|^2 + 2 \beta \varepsilon'}{\eta (1 - \varepsilon')}\right\}} dz' \\
        &\geq \alpha^2 {} \int_{B_{\bR^d}(z, \varepsilon' / 2)} \exp{-\bigg\{}\\&\quad\quad\frac{(|x - x^{\varepsilon'}| + |x^{\varepsilon'} - z'|)^2 + (|z' - y^{\varepsilon'}| + |y^{\varepsilon'} - y|)^2 + 2 \beta \varepsilon'}{\eta (1 - \varepsilon')}\bigg\} dz' \\
        &\geq \alpha^2 {} \int_{B_{\bR^d}(z, \varepsilon' / 2)} \exp{-\left\{\frac{(|x^{\varepsilon'} - z'| + \varepsilon)^2 + (|z' - y^{\varepsilon'}| + \varepsilon)^2 + 2 \beta \varepsilon'}{\eta (1 - \varepsilon')}\right\}} dz' \\
    \end{split}
    \end{equation}
    Now since, within the domain of integration, $|x^{\varepsilon'} - z'| \leq |x^{\varepsilon'} - z| + \frac{1}{2} \varepsilon' \leq \frac{1}{2}|x-y| + \frac{1}{2} \varepsilon' \leq \frac{1}{2}|x-y| + \frac{1}{2} \varepsilon$, and similarly for $y^{\varepsilon'}$, we have 
    \begin{equation}\label{eq:uniform-lb-eq-13}
    \begin{split}
        p^r_\eta(1, x, y) &\geq \alpha^2 {} \int_{B_{\bR^d}(z, \varepsilon' / 2)} \exp{-\left\{\frac{(\frac{1}{2}|x - y| + 2 \varepsilon)^2 + (\frac{1}{2}|x - y| + 2 \varepsilon)^2 + 2 \beta \varepsilon'}{\eta (1 - \varepsilon')}\right\}} dz' \\
        &= \alpha^2 {} \int_{B_{\bR^d}(z, \varepsilon' / 2)} \exp{-\left\{\frac{(|x - y| + 4 \varepsilon)^2 + 4 \beta \varepsilon'}{2 \eta (1 - \varepsilon')}\right\}} dz' \\
        &\geq \alpha^2 {} \int_{B_{\bR^d}(z, \varepsilon' / 2)} \exp{-\left\{\frac{(|x - y| + 4 \varepsilon)^2 + 4 \beta \varepsilon}{2 \eta (1 - \varepsilon)}\right\}} dz' \\
        &= \alpha^2 \operatorname{vol}_{\bR^d}\Big(B_{\bR^d}\big(0, \varepsilon' / 2\big)\Big) {} \exp{-\left\{\frac{(|x - y| + 4 \varepsilon)^2 + 4 \beta \varepsilon}{2 \eta (1 - \varepsilon)}\right\}} \\
        &\geq \tilde{\alpha} {} \exp{-\left\{\frac{|x - y|^2 + 8 \varepsilon \operatorname{diam}(D) + 16\varepsilon^2 + 4 \beta \varepsilon}{2 \eta (1 - \varepsilon)}\right\}} \\
        &\geq \tilde{\alpha} {} \exp{-\left\{\frac{|x - y|^2 + \tilde{\beta} \varepsilon}{2 \eta (1 - \varepsilon)}\right\}},
    \end{split}
    \end{equation}
    for appropriately chosen $\tilde{\alpha}, \tilde{\beta}$.
    This proves the claim.
\end{proof}
With Propositions \ref{prop:davies329} and \ref{prop:uniform-lb} in hand, it only remains to combine them to reach the main goal of this section: a proof of Theorem \ref{thm:reflected-ceta-uniform-convergence}.

\begin{proof}[Proof of Theorem \ref{thm:reflected-ceta-uniform-convergence}]
Take $\xi > 0$.
From the uniform lower bound, we have that for each $\varepsilon \in (0, \frac{1}{2})$, there is an $\eta_0 > 0$ such that for all $x,y \in \close{D}$,
\begin{equation}\label{eq:lb-to-limsup}
\begin{split}
    -\eta \log p^r_\eta(1, x, y) &\leq -\eta \log\left(\alpha_D \exp{-\left\{\frac{|x - y|^2 + \beta_D \varepsilon}{2 \eta (1 - \varepsilon)}\right\}}\right) \\
    &= \frac{|x - y|^2 + \beta_D \varepsilon}{2 (1 - \varepsilon)} -\eta \log\alpha_D.  \\
\end{split}
\end{equation}
The first term can be set arbitrarily close to $\frac{|x-y|^2}{2}$ by taking $\varepsilon$ small, and the second goes to zero with $\eta$.
Hence, this upper bound converges uniformly in the sense that for small enough $\eta$, it holds for all $x,y \in \close{D}$ that
\begin{equation}\label{eq:eta-log-reflected-ub}
    -\eta \log p^r_\eta(1, x, y) \leq \frac{|x-y|^2}{2} + \xi.
\end{equation}

Correspondingly, from the uniform lower bound in Proposition \ref{prop:davies329}, we have
\begin{equation}\label{eq:ub-to-liminf}
\begin{split}
    -\eta \log p^r_\eta(1, x, y) &\geq -\eta \log\left(c_{\delta, D} \exp{\left\{-\frac{|x-y|^2}{2(1 + \delta)\eta}\right\}}\right) \\
    &= \frac{|x-y|^2}{2(1 + \delta)} -\eta \log c_{\delta, D}. \\
\end{split}
\end{equation}
Since $\delta$ can be chosen small, and the second term goes to zero, 
we have for small enough $\eta$ that $-\eta \log p^r_\eta(1, x, y) \geq \frac{|x-y|^2}{2} - \xi$.
Together with the upper bound \eqref{eq:eta-log-reflected-ub}, this shows the uniform convergence, since $\xi > 0$ is arbitrary.
\end{proof}

As mentioned earlier, the connectedness of $\interior{D}$, and the nice differentiability properties of $c(x,y) = \frac{1}{2} |x-y|^2$, mean that Assumption \ref{ass:bernton44mod}, used in Theorem \ref{thm:weak-type-LDP}, is also fulfilled.
Having now proved Theorem \ref{thm:reflected-ceta-uniform-convergence}, in this section we have shown that the reflected Brownian SBs on $D$ satisfy the large deviation principle in Theorem \ref{thm:weak-type-LDP} and the strengthened version Corollary \ref{cor:katocor31mod}.

\section{Discussion and future work}

We have shown, in Section \ref{sec:ldp-uniform-convergent-cost}, an LDP for families of static SBs with a scaling parameter $\eta$, that satisfy a simple convergence criterion.
In Section \ref{sec:small-noise-reflected-browian-motion}, we give an example of such a scaled family from the field of generative modeling, namely the reflected Brownian SBs used in \cite{deng2024reflected}. 
This gives a partial positive answer to one of the two open questions posed at the end of the introduction in \cite{bernton2022entropic}, asking whether their large deviation results for $\varepsilon$-regularized EOT plans can be extended to SBs for general sequences $\{ R_\eta \}_{\eta > 0}$.
It is still an open question how much further this principle can be generalized.
The uniform convergence mode of $c_{\eta} \overset{}{\to} c$ may perhaps be weakened to a more permissive one; an immediate generalization, on locally compact spaces, is uniform convergence on compact subsets of $\calX \times \calY$.
Additionally, other types of dynamics --- SDEs with drifts and/or jumps, reflections on more complicated domains, etc. --- with a scaling parameter $\eta$, could be considered and analyzed similar to Section~\ref{sec:small-noise-reflected-browian-motion}.

Another line of future work is to establish weaker conditions under which the \textbf{dynamic} SBs also follow a large deviation principle on the path space $\calC_1$.
Research in this direction has recently started with \cite{Kato24}, where scaled Brownian motion is considered as the reference dynamics.
The proof does not generalize to much more involved dynamics, as it relies heavily on the Gaussianity of the bridge measures $R_\eta^{\cdot}$ on $\calC_1$.
In particular, it seems difficult to adopt to something like reflected Brownian motion.

\section*{Declarations}

\subsection*{Funding}

The research of VN and PN was supported by Wallenberg AI, Autonomous Systems and Software Program (WASP) funded by the Knut and Alice Wallenberg Foundation. The research of PN was also supported by the Swedish Research Council (VR-2018-07050, VR-2023-03484).

\subsection*{Conflict of interest}

The authors have no relevant conflict of interest to disclose.

\bibliographystyle{plain}
\bibliography{references}

\end{document}